\documentclass[11pt, reqno]{amsart} 
\usepackage[inner=2cm, outer=2cm, top=2cm, bottom=2cm, headsep=0.8cm, footskip=1.5cm]{geometry}
\usepackage{xcolor}
\usepackage{soul}
\usepackage{comment}
\usepackage{amsmath,amssymb}
\usepackage[cal=boondox,scr=boondoxo]{mathalfa}
\usepackage{bigints}
\usepackage{todonotes}
\presetkeys{todonotes}{color=orange!30}{}

\usepackage{amsmath} 

\usepackage{hyperref} 

\newtheorem{theo}{Theorem}[section]
\usepackage{yfonts}

\newtheorem{lemma}{Lemma}[section]
\newtheorem{prop}{Proposition}[section]
  \newtheorem{rem}{Remark}[section]
\newtheorem{definition}{Definition}[section]
\newtheorem{exa}{Example}[section]

\DeclareMathAlphabet{\mathpzc}{OT1}{pzc}{m}{it} 

\newcommand{\be}{\begin{eqnarray*}}
\newcommand{\ee}{\end{eqnarray*}}
\newcommand{\ben}{\begin{eqnarray}}
\newcommand{\een}{\end{eqnarray}}

\def \R{\mathbb R}
\def \N{\mathbb N}

\def \B{\mathbb B}

\def \A{\mathcal A}

\def \T{\mathcal T}
\def\V{\mathcal{V}}
\def\W{\mathcal{W}}
\def\F{\mathcal{F}}

\def \vsm{\vskip 0.2 truecm}
\def \ds{\displaystyle}

\def\bel{\begin{equation}\label}
\def\eeq{\end{equation}}
\def \w{\omega}
\def \xb{\check x_0}

\begin{document}
\title[Abnormality, controllability, and gap]{Nondegenerate abnormality, controllability,  and gap phenomena  in optimal control  with state constraints} 
\thanks{This research is partially supported by the  INdAM-GNAMPA Project 2020 ``Extended control problems: gap, higher order conditions and   Lyapunov functions"  and by the Padua University grant SID 2018 ``Controllability, stabilizability and infimum gaps for control systems'', prot. BIRD 187147.}

\author{Giovanni Fusco}\address{G. Fusco, Dipartimento di Matematica,
Universit\`a di Padova\\ Via Trieste 63, Padova  35121, Italy\\
email:\,
fusco@math.unipd.it}
\author{Monica Motta}\address{M. Motta, Dipartimento di Matematica,
Universit\`a di Padova\\ Via Trieste 63, Padova  35121, Italy\\
Telefax: (+39) 049 827 1499,\,\, Telephone: (39)(49) 827 1368
email:\,
motta@math.unipd.it}

\begin{abstract}
In optimal control theory, {\em infimum gap} means that  there is a gap between the infimum  values of a given minimum problem and an extended problem, obtained by enlarging the set of original solutions and controls.  The   gap phenomenon  is somewhat ``dual"  to the problem of the {\em controllability} of the original  control system  to an extended solution. In this paper we present sufficient conditions  for the absence of an infimum gap and for controllability for a wide  class of   optimal control problems subject to endpoint and state constraints. These conditions are based on a  nondegenerate version of the  nonsmooth constrained maximum principle, expressed in terms of subdifferentials.  In particular,  under some new constraint qualification conditions, we prove that:  (i) if an  extended minimizer is a  {\em nondegenerate normal} extremal,
then no gap shows up; (ii) given an extended solution verifying the constraints,  either it is a {\em nondegenerate abnormal} extremal, or 
the original system is controllable to it.   An application to the impulsive extension of a free end-time, non-convex  optimization problem with control-polynomial dynamics illustrates the results. 
\end{abstract}

\maketitle

\section{Introduction}
In the Calculus of Variations and in the Theory of Optimal Control is a rather common procedure to enlarge the space of solutions for those problems that do not admit a solution in a, say, ordinary space. Of course, a fundamental requirement for a good extension is that there is no  gap between the infimum of the original problem and that of the extended problem.
However, even if the set of ordinary solutions is $C^0$-dense in the set of extended trajectories,  in the presence of constraints   an infimum gap does in general occur, whenever all ordinary solutions in a $C^0$-neighborhood of a feasible  extended  trajectory (a local extended minimizer, for instance) violate the constraints.   In this case, we will refer to the  extended trajectory as {\em isolated}. By defining the   original   control system   {\em controllable} to an  extended trajectory whenever the trajectory is not isolated, we see that   gap avoidance and controllability are   strictly related issues.    Since Warga's early works  \cite{warga83_1,warga83_2}, it has emerged that  the existence of an infimum gap, or better, following our terminology, the fact that an extended trajectory is isolated,  is related to the validity of a maximum principle in {\em abnormal} form (as customary, abnormality  means that the scalar multiplier associated to the cost  is zero).  In particular, results of this kind have been obtained for the classical extension by relaxation (convex \cite{PV1,PV2} or in measure \cite{warga83_2,Kas99}) and, more recently, for the impulsive extension of control-affine systems with unbounded controls, with    or without state constraints (see \cite{FM120},  \cite{MRV}, respectively). Let us also mention   \cite{PR20}, where a general extension  is considered, but in the absence of state constraints and for smooth data.
These results are obtained by different techniques, which essentially reflect two different approaches to the maximum principle:  approach (a),   based on  the construction of approximating cones to reachable sets and on set separation arguments  \cite{warga83_1,warga83_2,Kas99,PR20}; approach (b), which   makes use of perturbation  and penalization techniques and of the Ekeland’s variational principle  \cite{PV1,PV2,MRV,FM120}.  For nonsmooth   optimal control problems, methods (a), (b) are  not easily comparable, as they require different assumptions on   dynamics and target but, above all,  lead to different abnormality conditions,  which involve the `derivative containers' introduced in   \cite{warga83_2} or the `quasi-differential quotients' defined in 	\cite{PR20}  in case (a), while in case (b) one uses a  by now   standard form of the nonsmooth constrained maximum principle due to Clarke,  expressed in terms of subdifferentials (see \cite{C90}).
  
\vsm 
The main purpose of this paper is to extend approach (b), applied so far only to particular cases, to identify under which general assumptions for the extension of an optimal control problem the following  statement is valid: {\em an isolated extended trajectory is an abnormal extremal.}  Furthermore,  we  give sufficient conditions  for  which we prove the stronger result: {\em an isolated extended trajectory is   an abnormal extremal of a nondegenerate version of the maximum principle. }

 \vsm
Precisely,  we consider the optimization problem  
$$
(P)
\left\{
\parbox[c][2cm]{.92\textwidth}{%
\begin{align}
&\qquad\qquad\text{minimize} \,\,\,   \Psi(y(S)) \qquad\qquad\qquad\qquad\qquad\qquad\qquad\qquad\qquad\qquad\text{ }\nonumber
 \\
&\text{over 
$(\w,\alpha,y)\in \V(S)\times\A(S) \times W^{1,1}([0,S];\R^n)$, verifying} \nonumber
\\
&\dot y(s)=\F(s,y(s),\w(s),\alpha(s)) \ \text{a.e.,}  \quad y(0)=\xb, \nonumber  \\
&h(s,y(s))\le0 \quad \forall s\in[0,S], \qquad 
y(S)\in\T, \nonumber
\end{align}}
\right.
$$
 where $\V(S):=L^1([0,S];V)$, $\A(S):=L^1([0,S];A)$,  
and the {\em extended} optimization problem, say $(P_e)$,  which is obtained by $(P)$   replacing in the minimization the control set $\V(S)$ with the larger set   $\W(S):=L^1([0,S];W)$, where $W=\overline{V}$. The data comprise the functions $\Psi:\R^n\to\R$, $\F:\R\times\R^n\times W\times A \to\R^n$, $h:\R\times\R^n\to\R$, the bounded set  $V\subset\R^m$, the compact set $A\subset\R^q$,  and the closed set $\T\subset\R^n$. 
 We  refer to any triple  $(\w,\alpha,y)$, where  $(\w,\alpha)\in \W(S)\times\A(S)$  and $y$ solves  the  Cauchy problem
 \bel{E}
 \dot y(s)=\F(s,y(s),\w(s),\alpha(s)) \quad  \text{a.e.,}  \qquad y(0)=\xb,
 \eeq
 as   {\em extended process}  or simply {\em process}. A process $(\w,\alpha,y)$ is  {\em feasible}  if  $h(s,y(s))\le0$   for all $s\in[0,S]$  and $y(S)\in\T$. 
The processes  $(\w,\alpha,y)$ of $(P)$, where   $\w\in \V(S)$,  will be called {\em strict sense processes}.   
 
As further extension, we consider the convex relaxation of $(P_e)$:
$$
(P_r)
\left\{
\parbox[c][2.8cm]{.915\textwidth}{%
\begin{align}
&\qquad\qquad\text{minimize} \,\,\,   \Psi(y(S)) \nonumber
 \\
 &\text{over  
$(\underline{\w},\underline{\alpha},\lambda, y)\in \W^{1+n}(S)\times\A^{1+n}(S)\times  \Lambda_n(S)   \times W^{1,1}([0,S]; \R^n)$, verifying }  \nonumber \\
&\dot y(s)=\sum_{k=0}^n\lambda^k(s)\F(s,y(s),\w^k(s),\alpha^k(s)) \  \text{a.e.,} \quad
y(0)=\xb,\nonumber \\ 
& h(s,y(s))\le0 \quad \forall s\in[0,S], \qquad  
y(S)\in\T, \nonumber
\end{align}}
\right.
$$
where $ \Lambda_n(S) := L^1([0,S];\Delta_n)$ and $\Delta_n$ is the $n$-dimensional simplex:  
\[
\Delta_n:=\left\{\lambda=(\lambda^0,\dots,\lambda^n): \ \ \lambda^k\ge0, \  k=0,\dots,n, \ \  \sum_{k=0}^n\lambda^k=1\right\}.
\]
Problem   $(P_r)$ is briefly referred to as  the  {\em relaxed} problem  and a   process   $(\underline{\w},\underline{\alpha},\lambda, y)$ for $(P_r)$ is  referred to as {\em relaxed} process.   We will identify a   process $(\w,\alpha,y)$  with any relaxed process $(\underline{\w},\underline{\alpha},\lambda, y)$, where $\underline{\w}:=(\w,\dots,\w)$,  $\underline{\alpha}:=(\alpha,\dots,\alpha)$.
 
The controls $a$ and $w$ play a different role, as only the control set $\V(S)$ to which $w$ belongs is extended. This distinction is reflected by the hypotheses on the dynamics $\F$. Referring to Section \ref{S1} for details, we observe that while continuity of $\F$ in $a$ will be enough, with respect to $w$ a   form of uniform continuity will be needed, both of $\F$ and of its Clarke subdifferential $\partial^c_x\F$. Moreover, not only $\overline{V}=W$, but there must also exist an increasing sequence of closed subsets $V_i\subseteq V$ such that $\cup_iV_i=V$ (in Remark \ref{RKas} below we will discuss   possible extensions of this hypothesis).
This formulation of the problem  includes as special cases both the  extension by convex relaxation considered in \cite{PV2}  (if $\F$ does not depend on $w$), and the impulsive, in general non-convex,  extension investigated in \cite{MRV,FM120}. In fact, if the isolated process that we show to be abnormal extremal belongs to a subclass of relaxed processes (for example, it is an extended process),   as a corollary, the  normality guarantees that there is no gap between $\inf\Psi$ over feasible strict sense processes and   on that subclass. 
\vsm
In Theorem \ref{Th1} of Section \ref{S1},  we state our first main result,  that  any isolated feasible relaxed  process  is an abnormal extremal.   The relevance of this result lies, in fact,  in its consequences, which are: (i)  a `normality test'  for no gap, namely,   if for a (local)   minimizer $\bar z$ of $(P_r)$ or $(P_e)$  the cost multiplier is $\ne0$  for any set of multipliers in the  maximum principle, at $\bar z$ there is no (local)  infimum gap; (ii)  the original control   system \eqref{E}  is controllable to any feasible relaxed process which is not an   abnormal  extremal  (see Theorems \ref{Cor_Norm}, \ref{Cor_Contr} below).

When the state constraint is active at the initial point $(0,\xb)$, it is well-known that sets of degenerate multipliers   such that {\em any} feasible relaxed trajectory is abnormal may exist,   making   the above  results (i), (ii) in fact useless.   This `degeneracy question' seems to have been disregarded in the literature on the relationship between gap and normality, apart from \cite{FM120}, where, however,  conditions are introduced which are never met in the case of a fixed initial point.  

Based on the above considerations, in Section \ref{S2} we provide a  condition inspired by the nondegeneracy conditions proposed in  \cite{FeV94,FeFoV99} ({\bf (H4)} below),  under which we refine the results of Section  \ref{S1}.   In particular, we establish that  any feasible relaxed  process which is isolated, is an abnormal extremal for a {\it nondegenerate} maximum principle, and derive as corollaries a  `nondegenerate normality test'  for no gap and a  `nondegenerate controllability condition'  (see Theorems \ref{Th2},  \ref{Cor_Norm2}, \ref{Cor_Contr2} below).

The   `normality' and the  `nondegenerate normality test'  are useful especially because in certain situations they allow  to deduce the absence of   gap from easily verifiable conditions, in the form of  constraint and endpoint  qualification conditions for normality,   on which there is a wide literature (see e.g. \cite{FoFr15,FrTo13,LFodaP11,AK20} and  references therein). As shown in \cite{MRV,MS20,FM120}, where some explicit normality sufficient conditions for the control-affine impulsive extension  are provided, these conditions  are in general weaker than those previously obtained  to get  the absence of   gap directly, as in \cite{AMR15,M18}.    
 
In Section \ref{SFT} we extend the previous results to  free end-time  optimal control problems.  We limit ourselves to considering the case of Lipschitz continuous time dependence, leaving the case of measurable time dependence to future investigations. Actually, Lipschitz continuous time dependence always arises in the impulsive extension of nonlinear problems with unbounded controls under the graph-completion approach, to which we apply our results in Section \ref{S3}. Impulsive optimal control problems  have been extensively studied together with their applications, mostly in the case of  control-affine systems, starting from   \cite{Ris:65,War:65,BR:88,Mi:94,MR:95}. We focus instead on the  less investigated case of control-polynomial dynamics \cite{RS00,MS14}. Among applications for which the polynomial dependence is relevant let us mention Lagrangian mechanical systems, possibly with friction forces, where inputs are identified with the derivatives of some coordinates. In this case, the degree of the polynomial is 2, as a consequence of the fact that the  kinetic energy is a  quadratic form of the velocity (see, e.g. \cite{AB1,BR10}).

\subsection{Notations and preliminaries}\label{sub1.1}
Given an interval $I\subseteq\R$ and a set $X \subseteq \R^k$, we write $W^{1,1}(I;X)$ 
 $L^1(I;X)$, $L^\infty (I;X)$, 
for the space of absolutely continuous functions
 Lebesgue integrable functions, essentially bounded functions 
 defined on $I$ and with values in $X$, respectively. When codomain or domain is clear, we will  use $\| \cdot \|_{L^1(I)}$, $\| \cdot \|_{L^{\infty}(I)}$,  or also  $\| \cdot \|_{L^1}$,  $\| \cdot \|_{L^{\infty}}$  to denote the $L^1$ and the ess-sup norm, respectively. 
  For all the classes of functions introduced so far, we will not specify domain and  codomain when the meaning is clear.   Furthermore,  we denote by $\ell(X)$, ${\rm co}(X)$, Int$(X)$, $\overline{X}$, $\partial X$ the Lebesgue measure, the convex hull, the interior, the closure, and  the boundary of $X$, respectively. As customary,  $\chi_{_X}$ is the characteristic function of $X$, namely $\chi_{_X}(x)=1$ if $x\in X$ and $\chi_{_X}(x)=0$ if $x\in\R^k\setminus X$.
For any subset $Y\subset X$,   $\text{{\rm proj}}_{Y}X$ will denote the  projection of $X$ on $Y$. We denote the closed unit ball in $\R^k$ by $\B_k$, omitting the dimension when it is clear from the context. Given a closed set ${\mathcal O} \subseteq \R^k$  and a point $z \in \R^k$, we define the distance of $z$ from ${\mathcal O}$ as $d_{\mathcal O}(z) := \min_{y \in {\mathcal O}} |z-y|$.
 We set $\R_{\geq 0} := [0,+\infty[$. For any  $a,b \in \R$, we write $a \vee b:= \max \{a,b\}$.
  We use $NBV^{+}([0,  S];\R)$ to denote the space of increasing, real valued functions $\mu$ on $[0, S]$ of bounded variation, vanishing at the point 0 and right continuous on $]0, S[$.   Each  $\mu\in NBV^+([0,S];\R)$  defines a Borel measure on $[0, S]$,   still denoted by $\mu$,  its total variation function is  indicated by  $\| \mu \|_{TV}$ or   by $\mu([0, S])$, and   its support  is spt$\{\mu\}$.  
 
Some standard constructs from nonsmooth analysis are employed in this paper. For background material we refer the reader for instance to \cite{C90,OptV}. A set 
$K \subseteq \R^k$ is a {\em cone} if $\alpha k \in K$ for any $\alpha >0$,  whenever $k \in K$. Take a closed set $D \subseteq \R^k$ and a point $\bar x \in D$, 
the \textit{limiting normal cone} $N_D(\bar x)$ of $D$ at $\bar x$ is given by 
\[ 
N_D(\bar x) := \left\{  \eta\in\R^k \text{ : } \exists x_i \stackrel{D}{\to} \bar x,\, \eta_i \to \eta \,\,\text{ such that }\,\,  \limsup_{x \to x_i} \frac{\eta_i \cdot (x-x_i)}{|x-x_i|}  \leq 0   \ \ \forall i     \right\},
\]
 in which the notation $x_{i} \stackrel{D}{\longrightarrow}\bar{x}$ is used to indicate that  all points in the converging sequence  $(x_i)_{i}$  lay in $D$.  
Take a lower semicontinuous function  $G:\R^k \to \R$  and a point $\bar x \in \R^k$, the \textit{limiting subdifferential} of $G$ at $\bar x$ is 
\[
\partial G(\bar x) := \left\{     \xi  \text{: } \exists   \xi_i \to \xi, \, x_i \to \bar x \text{ s.t. } \limsup_{x \to x_i} \frac{\xi_i \cdot (x-x_i) - G(x) + G(x_i)}{|x-x_i|}  \leq 0   \ \ \forall i  \right\}.
\]
 If  $G:\R^k\times\R^{h} \to \R$ is a lower semicontinuous function and $(\bar x,\bar y) \in \R^k\times\R^{h}$, we  write $\partial_x G(\bar x,\bar y)$, $\partial_y G(\bar x,\bar y)$ to denote the {\em partial limiting subdifferential of $G$ at  $(\bar x,\bar y)$ w.r.t. $x$, $y$}, respectively. When $G$ is differentiable, $\nabla G$ is  the usual gradient operator and $\nabla_x G$, $\nabla_y G$ denote the  partial derivatives of $G$.
 Given a locally Lipschitz continuous function $G:\R^k\to\R$ and $\bar x \in \R^k$, the \textit{reachable hybrid subdifferential} of $G$ at $\bar x$ is
\[
\partial^{*>} G(\bar x) := \, \left\{    \xi    \text{: } \exists (x_i)_i  \subset \text{diff}(G) \setminus \{ \bar x\} \text{ s.t. }  x_i \to \bar x, \, G(x_i) >0 \,\,\forall i, \    \nabla G(x_i) \to \xi       \right\}
\] 
while the \textit{reachable gradient} of $G$ at $\bar x$ is 
\[
\partial^{*} G(\bar x) := \, \left\{    \xi   \text{: } \exists (x_i)_i  \subset \text{diff}(G) \setminus \{ \bar x\} \text{ s.t. }  x_i \to \bar x  \text{ and }   \nabla G(x_i) \to \xi       \right\}
\] 
where diff($G$) denotes the set of differentiability points of $G$.
We define the \textit{hybrid subdifferential} as $\partial^{>} G(\bar x):=${\rm co}$\,\partial^{*>} G(\bar x)$. The set $\partial^*G(\bar x)$ is nonempty, closed, in general non convex, and its convex hull coincides with  the Clarke subdifferential $\partial^{c} G(\bar x)$, that is, $\partial^{c} G(\bar x)=${\rm co}$\,\partial^{*} G(\bar x)$. Finally, when $G$ is locally Lipschitz continuous, $\partial^{c} G(\bar x)=${\rm co}$\,\partial  G(\bar x)$. With a small abuse of notation, given   a locally Lipschitz continuous function $G:\R^k\to\R^l$ and $\bar x \in \R^k$, we still write $\partial^{c} G(\bar x)$ to denote the {\em Clarke generalized Jacobian}, defined as 
\[
\partial^{c} G(\bar x) :=\text{\,co\,} \, \left\{    \xi   \text{: } \exists (x_i)_i  \subset \text{diff}(G) \setminus \{ \bar x\} \text{ s.t. }  x_i \to \bar x  \text{ and }   \nabla G(x_i) \to \xi       \right\},
\] 
where now $\nabla G$ denotes the classical Jacobian matrix of $G$.  
 

\section{Infimum gap, isolated processes  and  abnormality}\label{S1} 

In the following, when the final time $S>0$ is clear from the context, we simply write  $\V$, $\W$,   $\A$, $\Lambda_n$, instead of  $\V(S)$, $\W(S)$, $\A(S)$, $\Lambda_n(S)$,  respectively.

\subsection{Basic assumptions} We  shall consider  the following hypotheses, in which $(\underline{\bar\w},\underline{\bar\alpha}, \bar\lambda,\bar y)$ is  a   feasible relaxed process, 
 which we call the {\em reference process} and, 
 for some $\theta>0$,   we set
$$
\Sigma_{\theta}:=\left\{(s,x)\in\R\times\R^{n}: \ \ s\in[0,S], \ x\in \bar y(s)+\theta\,\B\right\}.
$$ 

{\em  \begin{itemize} 
\item[{\bf (H1)}] 
There exists a sequence $(V_i)_i$ of closed subsets of $V$ such that $V_i\subseteq V_{i+1}$ for every $i$ and  $\bigcup_{i=1}^{+\infty}V_i=V$.

\item[{\bf (H2)}]  The constraint function  $h$ is  upper semicontinuous  and $K_{_h}$-Lipschitz continuous in $x$, uniformly w.r.t. $s$ in $\Sigma_{\theta}$.
\item[{\bf (H3)}] 
{\rm (i)} For all  $(x,w,a)\in \{x\in\R^n: \ (s,x)\in \Sigma_{\theta} \ \text{ for some $s\in[0,S]$}\}\times W\times A$,  $\F(\cdot,x,w,a)$ is Lebesgue measurable on $[0,S]$ and for any $ (s,x)\in \Sigma_{\theta}$,    $ \F(s,x,\cdot,\cdot)$ is continuous on $W\times A$.  Moreover, there exists   $k\in L^1([0,S];\R_{\geq 0})$ such that, for all $(s,x,w,a)$, $(s,x',w,a)\in \Sigma_{\theta}\times W\times A$, we have
\[
|\F(s,x,w,a)|\le k(s), \qquad |\F(s,x',w,a)-\F(s,x,w,a)|\le k(s)|x'-x|.
\]
{\rm (ii)} There exists some continuous   increasing function $\varphi:\R_{\ge0}\to\R_{\ge0}$ with $\varphi(0)=0$ such that for any $(s,x,a)\in \Sigma_{\theta}\times A$,   we have
\[
\begin{array}{l}
|\F(s,x,w',a)- \F(s,x,w,a)| \leq k(s)\varphi(|w'-w|) \qquad\forall w',\,w\in W, \\
\partial_x^c \F(s,x,w',a)\subseteq \partial_x^c \F(s,x,w,a)+k(s)\varphi(|w'-w|)\,\B \qquad\forall w',\,w\in W.
\end{array}
\]
\noindent When hypothesis {\bf (H3)} is valid for $k \equiv K_{_\F}$  for some constant $K_{_\F} >0$, we will refer to  {\bf (H3)} as {\bf (H3)$'$}.
\end{itemize}}

\begin{rem}\label{Rhyp12} {\rm Condition  {\bf (H1)} implies (and  is in general stronger  than) the density of $\V$ in $\W$ in the $L^1$-norm. Indeed, given $\w\in\W$, from  {\bf (H1)} it follows that  for any  $\delta>0$ there exists some index $i_{\delta}$, such that the Hausdorff distance $d_H(V_i,W)<\delta/S$ for every $i\ge i_{\delta}$. Hence,  by the selection theorem  \cite[Theorem 2, p. 91]{AC84}    for such $i$ there is some measurable function $\w_i(s)\in \text{proj}_{V_i}(\w(s))$ for a.e. $s$, which thus verifies $\|\w_i -\w\|_{L^1}\le S\|\w_i -\w\|_{L^\infty}\le Sd_H(V_i,W)<\delta$. 
  In particular, when  $\V= L^1([0,S];V)$ for some  subset $V\subset W$ such that int$(W)\subseteq V\subseteq W$ and $W= \overline{\text{int}(W)}$, the validity of {\bf (H1)} follows  by elementary properties of closed and open  subsets of $\R^n$. }
\end{rem} 
\begin{rem}\label{RKas} {\rm  As one can easily deduce from the proofs in Section \ref{S4} below, condition  {\bf (H1)} could be replaced by the hypothesis that there exists a subset  $\V\subset \W:= L^1([0,S]; W)$ which is  closed by finite concatenation and verifies:  

\noindent (i) {\em there exists an increasing sequence of closed subsets   $(\V_i)_i\subseteq \V$ such that  $\cup_{i} \V_i = \V$ and,   for any $w \in \W$ and  $\delta>0$, there are $i_\delta$ and $w_\delta \in \V_{i_\delta}$,  such that $\| w_\delta - w  \|_{L^1} \leq \delta$;}

\noindent (ii)  {\em for every  $i$, for  the optimization  problem obtained from $(P)$ by replacing $\V$ with  $\V_i$, a nonsmooth constrained maximum principle is valid. } 
 
For example, from \cite{Kas86}  a condition sufficient for (ii) to hold true is  the $C^0$-closure of the set of the solutions to \eqref{E}  as $(\w,\alpha)\in  \V_i\times\A$, for every $i$.
}
\end{rem} 
\begin{rem}\label{Rhyp5} {\rm Condition {\bf (H3)}(ii) is  satisfied, for instance,  when  $\F(s,x,w,a)=\F_1(s,x,a)+\F_2(s,x,w,a)$, where  $\F_1$, $\F_2$  verify hypothesis {\bf (H3)}(i) and, in addition, the function  $\F_{2}(s,\cdot,w,a)$ is $C^1$  and $\nabla_x\F_{2}$  is continuous on the compact set $\Sigma_{\theta}\times W\times A$. }  
\end{rem}

  \subsection{Infimum gap and isolated processes} 
  Let us write 
$\Gamma $, $\Gamma_e$,  $\Gamma_{r}$,
 to denote the sets of strict sense,    extended,   and    relaxed  processes which are feasible,
 respectively. 
 As  observed in the Introduction,   $\Gamma$, $\Gamma_e$, 
 can be identified   
 with   subsets  of $\Gamma_r$.  
 \begin{definition}[Minimizer] \label{mindef}
 A process  $\bar z:=(\underline{\bar\w},\underline{\bar\alpha},\bar\lambda, \bar y)\in\Gamma_g$, $g\in\{r,e\}$,  is  called a {\it local $\Psi$-minimizer} for  problem $(P_g)$  if,  for some $\delta>0$,  one has
 $$
\Psi(\bar y(S))=\inf\left\{\Psi(y(S)): \ \  (\underline{\w},\underline{\alpha},\lambda, y)\in\Gamma_g ,   \ \ \|y-\bar y\|_{L^\infty}<\delta \right\}.
$$
The process $\bar z$ is a $\Psi$-minimizer for  problem $(P_g)$ if $\ds\Psi(\bar y(S))=\inf_{\Gamma_g} \Psi(y(S))$.
\end{definition}
\begin{definition}[Infimum gap] \label{gapdef}
Fix    $\bar z:=(\underline{\bar\w},\underline{\bar\alpha},\bar\lambda, \bar y)\in \Gamma_r$. 

\noindent {\rm (i)} Let $\Psi:\R^n\to\R$ be a continuous function. When   there is some $\delta>0$ such that 
$$
 \ds \Psi(\bar y(S))  <\,  \inf \left\{ \Psi(y(S)): \ \ (\w,\alpha,y)\in\Gamma, \ \ \|y-\bar y\|_{L^\infty}<\delta \right\} 
 $$
(as customary, when the set is empty we set the infimum $=+\infty$), we say that  {\em at $\bar z$ there is a  local $\Psi$-infimum   gap}.  We  say that there is a  {\em $\Psi$-infimum gap} with problem $(P_g)$, $g\in\{r,e\}$,  if 
$
 \ds\inf_{\Gamma_g }\Psi(y(S))   <\,  \inf_{\Gamma}   \Psi(y(S)). 
 $
 
\noindent{\rm (ii)} 
When at $\bar z$ there is a   local $\Psi$-infimum  gap or if   there is a  $\Psi$-infimum   gap (with $(P_e)$ or  $(P_r)$) for some   $\Psi$ as above, we say that  {\em at $\bar z$ there is a   local infimum  gap} or that {\em there is an   infimum   gap}  with $(P_e)$ or  $(P_r)$, respectively.
   \end{definition}      
 Obviously,  a $\Psi$-infimum gap with $(P_e)$ implies  a $\Psi$-infimum gap with $(P_r)$, and   it may happen that  
 $
\ds \inf_{\Gamma_r }\Psi(y(S))   <\, \inf_{\Gamma_e }\Psi(y(S))   <\, \inf_{\Gamma}   \Psi(y(S)).
 $ 

The notion of local infimum gap at $\bar z$   is equivalent to 
the notion of {\em isolated process}, first introduced in \cite{MRV},  which is independent of any optimization problem.
 \begin{definition}[Isolated process  and controllability] \label{isolated}
We say that a process $\bar z:=(\underline{\bar\w},\underline{\bar\alpha},\bar\lambda, \bar y)\in\Gamma_r$ is    {\em isolated} (in $\Gamma$) if,  for some $\delta>0$, one has
$$
\left\{ (\w,\alpha,y)\in  \Gamma:  \ \ \|y-\bar y\|_{L^\infty}<\delta \right\}=\emptyset.
$$
The    control system  \eqref{E}  is  said  {\em controllable to $\bar z$} if   $\bar z$ is not isolated.
\end{definition}
 
\begin{prop} \label{iso}  
Let  $\bar z:=(\underline{\bar\w},\underline{\bar\alpha},\bar\lambda,\bar y)\in\Gamma_r$. Then
\begin{itemize}
\item[{\rm (i)}] if $\bar z$  is isolated, then at $\bar z$ there is a local infimum gap and, for   some $\delta>0$, one has   $\inf \big\{ \Psi(y(S)):$  $(\w,\alpha,y)\in\Gamma , \ \ \|y-\bar y\|_{L^\infty}<\delta \big\}=+\infty$ for every continuous  $\Psi$;
\item[{\rm (ii)}] if at   $\bar z$ there is a local infimum gap, then  $\bar z$ is isolated.
\end{itemize} 
As a consequence,   $\bar z$ is isolated if and only if at $\bar z$ there is a local infimum gap.
\end{prop}
\begin{proof} The proof of (i) is trivial, hence we limit ourselves to prove (ii). Suppose   that at   $\bar z$ there is a $\Psi$-local infimum gap for some continuous $\Psi$ and some $\delta>0$, but  $\bar z$ is not isolated. Then, for every $i\in\N$, $i\ge\frac{1}{\delta}$, there   exists  a feasible strict sense process $z_i= (\w_i,\alpha_i,y_i)\in\Gamma$ such that $\|y_i-\bar y\|_{L^\infty}<\frac{1}{i}$,  and,  by  the continuity of $\Psi$,  
$$
 \Psi(\bar y(S))<\inf \left\{ \Psi(y(S)): \,    (\w,\alpha,y)\in\Gamma,  \,  \|y-\bar y\|_{L^\infty}<\delta \right\}\le \lim_{i }\Psi(y_i(S)), 
 =\Psi(\bar y(S))
$$
which gives the desired contradiction.
\end{proof}
 Incidentally,   if at  $\bar z$   there is a  local $\Psi$-infimum gap for some $\Psi$, for some $\delta>0$ there is in fact a local $\tilde\Psi$-infimum gap and  $\inf \big\{ \tilde\Psi(y(S)): \   (\w,\alpha,y)\in\Gamma,   \ \|y-\bar y\|_{L^\infty}<\delta \big\}=+\infty$  for {\em every} continuous function $\tilde\Psi$. 

 \subsection{Abnormality and infimum gap} 
 We introduce a notion of normal and abnormal extremal for the relaxed optimization problem.
\begin{definition}[Normal and abnormal extremal] \label{Dnormal} 	 
Let  $\bar z:=(\underline{\bar\w},\underline{\bar\alpha}, \bar\lambda,\bar y)$ be a feasible relaxed  process.  Given a   function $\Psi:\R^n\to\R$ which is Lipschitz continuous on a neighborhood of $\bar y(S)$,      we say that $\bar z$  is  a  {\em  $\Psi$-extremal}  if there exist   a path $p \in W^{1,1}([0, S]; \R^n)$, $\gamma \geq 0$,  $\mu \in NBV^{+}([0, S];\R)$, $m: [0,S] \to \R^{n}$ Borel measurable and $\mu$-integrable function, verifying the following conditions:
$$
\parbox[c][5cm]{1.\textwidth}{%
\begin{align}
& \| p \|_{L^{\infty}}+ \| \mu \|_{TV} + \gamma \neq 0; \label{fe1} \\
& -\dot p(s)  \in \sum_{k=0}^n \bar\lambda^k(s)\,{\rm co}\,  \partial_{x}  \left\{q(s)\cdot \F(s, \bar y(s),  \bar \w^k(s),\bar\alpha^k(s))\right\} \  \text{ a.e.;} \label{ad_eq} \\
&  -q(S)  \in \gamma \partial \Psi\left(  \bar y(S)\right)  
 +N_{\T } ( \bar y(S) ); \label{trans_cond} \\
&\text{for every $k=0,\dots,n$,  for a.e. $s\in[0,S]$, one has } \nonumber \\
&q(s)\cdot \F\big(s, \bar y(s), \bar \w^k(s),\bar\alpha^k(s)\big) 
 =  \max_{(w,a)\in W\times A}  q(s)\cdot\F\big(s, \bar y(s), w,a\big);  
 \label{maxham} \\
 &m (s) \in \partial_{ x}^{>}\, h\left(s, \bar y(s)\right)  \qquad \text{$\mu$-a.e.;}
\label{m_eq} \\
& spt(\mu) \subseteq \{ s\in [0,S] \text{ : } h\left(s, \bar y(s)\right) = 0 \}, \label{spt_mu}
\end{align}}
$$
where
\[
 q(s)  := 
\begin{cases}
  p(s) + \int_{[0,s[} m(\sigma) \mu(d\sigma)  \qquad\,\, s\in [0,S[, \\
  p(S) + \int_{[0, S]} m(\sigma) \mu(d\sigma)  \qquad s=S.
\end{cases}
\] 
 We will call a  $\Psi$-extremal  \textit{normal} if all possible choices of   $(p,\gamma,\mu, m)$ as above have $\gamma >0$, and  \textit{abnormal}  when it is not normal. Since  the notion of abnormal $\Psi$-extremal is actually independent of  $\Psi$,  in the following    abnormal $\Psi$-extremals will be simply called {\em abnormal  extremals}.   \end{definition}
  
\begin{theo}\label{Th1} Let $\bar z:=(\underline{\bar\w},\underline{\bar\alpha},\bar\lambda, \bar y)$ be a feasible relaxed process.  Assume that hypotheses {\em {\bf (H1)}--{\bf (H3)}} are verified. If  at $\bar z$ there is a local infimum gap, then $\bar z$ is an abnormal extremal. 
\end{theo}

 A first remarkable  immediate consequence of Theorem \ref{Th1}  is the following sufficient condition for the absence of an infimum gap.
\begin{theo}\label{Cor_Norm}  
Suppose that there exists a [local]  $\Psi$-minimizer  $\bar z:=(\underline{\bar\w},\underline{\bar\alpha}, \bar\lambda,\bar y)$  for  $(P_e)$ or  $(P_r)$, for which  hypotheses {\em {\bf (H1)}--{\bf (H3)}} are verified, $\Psi$ is Lipschitz continuous in a neighborhood of $\bar y(S)$,  and  which is a normal $\Psi$-extremal. Then,  [at $\bar z$] there is no [local] $\Psi$-infimum gap with $(P_e)$ or  $(P_r)$, respectively.
\end{theo}

\begin{rem}\label{Rdim1}{\rm  By a  well known constrained maximum principle (see \cite[Ch. 9]{OptV}),  local $\Psi$-minimizers of   $(P_r)$  are $\Psi$-extremal  in a stronger form than in Definition \ref{Dnormal}, in which the costate differential inclusion \eqref{ad_eq} is replaced by
\bel{truead_eq} -\dot p(s)  \in \, {\rm co}\,  \partial_{x}\left\{\sum_{k=0}^n \bar\lambda^k(s) \, q(s)\cdot \F(s, \bar y(s),  \bar \w^k(s),\bar\alpha^k(s))\right\} \quad \text{for a.e. $s \in [0,S]$.}
\eeq
The need to consider  \eqref{ad_eq} derives from the perturbation technique used in the proof of Theorem \ref{Th1} (see also \cite{PV1}).
In fact, \eqref{ad_eq} may differ from \eqref{truead_eq} only in case of nonsmooth dynamics. Precisely, if $\F(s,\cdot, \bar\w^k(s),\bar\alpha^k(s))$ is continuously differentiable at $\bar y(s)$, for all $k=0,\dots,n$ and a.e. $s\in[0,S]$,  then both differential inclusions reduce to  the adjoint equation
$$
 -\dot p(s)=  \sum_{k=0}^n \bar\lambda^k(s) \, q(s)\cdot \F_x(s, \bar y(s),  \bar \w^k(s),\bar\alpha^k(s))  \quad \text{for a.e. $s \in [0,S]$.}
$$}
\end{rem}
%
%
%

Thanks to Proposition \ref{iso}, from Theorem \ref{Th1} we can also deduce the following sufficient condition for controllability to the reference trajectory.
\begin{theo}\label{Cor_Contr}  Let $\bar z:=(\underline{\bar\w},\underline{\bar\alpha},\bar\lambda, \bar y)$ be a feasible relaxed process and assume that hypotheses {\rm {\bf (H1)}--{\bf (H3)}} are verified. Then, either
\begin{itemize}
\item[{\rm (i)}] $\bar z$ is not isolated in $\Gamma$, namely, there exists a sequence of feasible processes $(\w_i,\alpha_i,y_i)\in\Gamma$ such that $\|y_i-\bar y\|_{L^\infty}\to 0$ as $i\to+\infty$; or
\item[{\rm (ii)}] $\bar z$ is an abnormal extremal.
  \end{itemize}
\end{theo}
 \begin{proof} Suppose by contradiction   statements (i),  (ii)   false, namely, let $\bar z$ be an isolated process which is  not an abnormal extremal. Then,  by Proposition \ref{iso} at $\bar z$ there is a local infimum gap   and  $\bar z$ should be an abnormal extremal by  Theorem \ref{Th1}. 
 \end{proof}

The proof of Theorem \ref{Th1} is given in Section \ref{S4}.

\section{Infimum gap and  nondegenerate abnormality}\label{S2} 
 Disregarding the endpoint constraint, when the state constraint is active at the initial time  there always exist  sets of degenerate multipliers, as, for instance, $\gamma=0$,  $\mu=\delta_{\{0\}}$,   $ p(s)= -m(0)\in \partial^>h(0,\xb)$ for all $s\in[0,S]$.\footnote{For any $r\in\R$, $\delta_{\{r\}}$ is the Dirac unit measure concentrated at $r$.}   
With the degenerate multipliers,  the maximum principle is clearly useless, not only to select  minimizers, but also to identify no-gap conditions in the form of a `normality test'. 

\begin{definition}[Nondegenerate normal and abnormal extremal]
 Assume that   $\bar z:=(\underline{\bar\w},\underline{\bar\alpha},\bar\lambda,\bar y)$ is a feasible relaxed  process.  Given a   function $\Psi:\R^n\to\R$ which is Lipschitz continuous on a neighborhood of $\bar y(S)$, let  $\bar z$  be  a $\Psi$-extremal.  We call {\em nondegenerate multiplier} any set of multipliers  $(p, \gamma, \mu,m)$   that meets the conditions   of Definition  \ref{Dnormal}, and also verifies the following strengthened nontriviality condition
\begin{equation} \label{nondeg}
\mu(]0,S]) +\Vert q \Vert_{L^{\infty}} + \gamma \neq 0, 
\end{equation}
where $q$ is   as in Definition  \ref{Dnormal}.
  We call $\bar z$   a {\em nondegenerate normal $\Psi$-extremal}   if all possible choices of nondegenerate multipliers  have $\gamma >0$, and  a  {\em nondegenerate abnormal $\Psi$-extremal}  when  there exists at least one nondegenerate multiplier with $\gamma =0$. Since  nondegenerate  abnormal $\Psi$-extremals do  not depend on  $\Psi$,   in the following  we will call them simply  {\em nondegenerate abnormal  extremals}.  
 \end{definition}
As it is   easy to see,  a nondegenerate abnormal extremal is always an abnormal extremal, and, on the contrary, any normal $\Psi$-extremal is also nondegenerate normal. However, we may have situations where  a nondegenerate normal $\Psi$-extremal, is not a normal $\Psi$-extremal,   as illustrated in Example \ref{exa}  below. 

\vsm
To introduce our  constraint qualification conditions, we define   $\Lambda_n^1\equiv\Lambda_n^1(S)$, as
\bel{lambda1}
\ds\Lambda_n^1 :=L^1([0,S];\Delta^1_n),  \ \Delta^1_n:=\cup_{k=0}^n\, \{e_k\} \ \ ( {e}_0,\dots,{e}_n  \text{ canonical basis of } \R^{1+n}),
\eeq
and extend the relaxed control system  by introducing a new variable, $\xi$. Precisely, for any $(\underline{\w},\underline{\alpha},\lambda) \in  \W^{1+n}\times\A^{1+n}\times \Lambda_n$, we denote by $(\xi,y)[\underline{\w},\underline{\alpha},\lambda]$  the solution  to
 \bel{Er}
\left\{\begin{array}{l}
(\dot\xi,\dot y)(s)=\Big(\lambda(s),\,  \sum_{k=0}^n\lambda^k(s)\F(s,y(s),\w^k(s),\alpha^k(s))\Big) \ \text{a.e.,} \\ 
(\xi,y)(0)=(0,\xb),
\end{array}\right.
\eeq
and with a small abuse of notation,  in the following  we  call $(\underline{\w},\underline{\alpha},\lambda,\xi,y)$ with $(\xi,y):=(\xi,y)[\underline{\w},\underline{\alpha},\lambda]$, a relaxed process.  Observe that,  when $(\underline{\w},\underline{\alpha})=(\w,\dots,\w,\alpha,\dots,\alpha)$ for some $(\w,\alpha)\in\W\times\A$,  the $y$-component of   the solution $(\xi,y)[\underline{\w},\underline{\alpha},\lambda]$ is in fact an extended trajectory, namely, it solves the original problem \eqref{E}. 

%
%

  \vsm

Let  $\bar z:=(\underline{\bar\w},\underline{\bar\alpha},\bar\lambda,\bar\xi,\bar y)$ be a feasible relaxed  process. Define the constraint set
$$
\Omega:=\{(s,x)\in\R^{1+n}:  \ \  h(s,x)\le0\}. 
$$
 We shall consider the following hypotheses. 
 \vsm
{\em  \begin{itemize} 
\item[{\bf (H4)}] If $(0,\xb)\in \partial \Omega$,  there exist some   $\tilde\delta>0$,  $\bar s\in]0,S]$,  sequences  $(\tilde \w_i,\tilde\alpha_i,  {\tilde\lambda}_i)_i\subset (\W\cap\V(\bar s))\times\A\times \Lambda_n^1$,   $(\hat \w_i,\hat\alpha_i)_i\subset \V(\bar s)\times\A(\bar s)$,  and $(\tilde r_i)_i\subset L^1([0,S];\R_{\ge0})$ with $\lim_{i\to+\infty}\|\tilde r_i\|_{L^1}=0$,  such that the following properties {\rm (i)}--{\rm (iv)}, where $(\tilde \xi_i,\tilde y_i):=(\xi,y)[\tilde \w_i,\dots,\tilde\w_i,\tilde\alpha_i,\dots,\tilde\alpha_i, {\tilde\lambda}_i]$, are verified.
\begin{itemize}
\item[{\rm (i)}] One has
\bel{convtilde}
\lim_{i\to+\infty}\|(\tilde \xi_i,\tilde y_i) -(\bar \xi,\bar y)\|_{L^\infty}=0; 
\eeq
\item[{\rm (ii)}] for every $i$, one has
\bel{htilde}
  h(s,\tilde y_i(s))\le 0 \qquad \forall s\in[0,\bar s]; 
 \eeq
\item[{\rm (iii)}]  for every $i$, there is a Lebesgue measurable subset $\tilde E_i\subset[0,S]$ such that  
 \bel{tildew}
 \begin{array}{c}
 (\tilde \w_i,\tilde\alpha_i,  \tilde\lambda_i)(s)\in \bigcup_{k=0}^n\{(\bar\w^k(s),\bar\alpha^k(s),e^k)\}+(\tilde r_i(s),0,0)\B,  \ \text{a.e. $s\in \tilde E_i$;} \\[1.5ex]
  \lim_{i\to+\infty}\ell(\tilde E_i)=S;
  \end{array}
 \eeq
\item[{\rm (iv)}] for every $i$ and for all $(\zeta_0, \zeta)\in\partial^* h(0, \xb)$, for a.e.  $s\in[0,\bar s]$  one has
  \bel{cqd}
\begin{array}{l}
\ds \zeta\cdot  \big[\F(s,\xb,(\hat \w_i, \hat\alpha_i)(s))-\F(s,\xb, (\tilde \w_i, \tilde\alpha_i)(s))\big]\leq-\tilde\delta.
\end{array}
\eeq
\end{itemize}
\end{itemize}}
\vsm

\begin{rem}\label{RH4}{\rm  Some comments on  hypothesis {\bf (H4)} are in order. 
\begin{itemize}
\item[(1)] It   prescribes additional conditions to   assumptions {\bf (H1)}--{\bf (H3)} only when the initial point $(0,\xb)$ lies on the boundary of the constraint set $\Omega$. Incidentally, this is not equivalent to having $h(0,\xb)=0$, as it may clearly happen that $h(0,\xb)=0$ but  $(0,\xb)\in$Int$(\Omega)$. 
\item[(2)] When $(0,\xb)\in \partial \Omega$, the first part of hypothesis {\bf (H4)} substantially requires the existence of  strict sense  processes that approximate the reference process and satisfy the state constraint on some (small) interval $[0,\bar s]$, with  controls which are close 
to   controls $(\bar\w_i,\bar\alpha_i,\bar\lambda_i)$ belonging to $ \bigcup_{k=0}^n\{(\bar\w^k(s),\bar\alpha^k(s),e^k)\}$ for a.e. $s\in[0,S]$. Let us point out that, disregarding  the state constraint \eqref{htilde},  the existence of approximating controls that satisfy the remaining conditions \eqref{convtilde}, \eqref{tildew} follows  by  the relaxation Theorem together with hypothesis {\bf (H1)}, as we will see in the proof of Theorem \ref{Th1} below, in Subsection \ref{SubTh1}. Relation  \eqref{cqd}, on the other hand, is an adaptation of   known constraint qualification conditions  (see e.g. \cite{FeV94,FeFoV99}), in which the reference (relaxed) control is replaced by  approximating strict sense controls  $(\tilde \w_i, \tilde\alpha_i)\in \V(\bar s)\times\A(\bar s)$. 

\item[(3)] If   hypotheses {\bf (H1)}, {\bf (H2)}, and {\bf (H3)$'$} with reference to $\bar z$ are verified, then in hypothesis {\bf (H4)} one can assume that  the control sequence  $(\hat \w_i,\hat\alpha_i)_i$  belongs to the extended control set $\W(\bar s)\times\A(\bar s)$ rather than  $\V(\bar s)\times\A(\bar s)$. Indeed, using the notation of  {\bf (H1)}--{\bf (H3)$'$}, let us choose some  $\rho>0$ such that   $K_h\,K_\F\varphi(\rho)\le \frac{\tilde\delta}{2}$,  and let $j\in\N$ verify $d_H(V_j,W)\le\rho$.  Hence, for every $i\in\N$  there exists a measurable selection $\hat \w_i^*(s)\in\text{proj}_{V_j}(\hat \w_i(s))$ for a.e. $s\in[0,S]$, such that $\|\hat \w_i^*-\hat \w_i\|_{L^\infty}\le\rho$  (see also Remark \ref{Rhyp12}), and,  for all $(\zeta_0, \zeta)\in\partial^* h(0, \xb)$   (by adding and subtracting  $\zeta\cdot \F(s,\xb,(\hat \w_i^*, \hat\alpha_i)(s))$), one has
$$
\begin{array}{l}
 \zeta\cdot  \big[\F(s,\xb,(\hat \w_i^*, \hat\alpha_i)(s))-\F(s,\xb, (\tilde \w_i, \tilde\alpha_i)(s))\big]\leq
 -\frac{\tilde\delta}{2}, \quad  \text{a.e.  $s\in[0,\bar s]$},
\end{array}
$$
as soon as  $(\hat \w_i,\hat\alpha_i)$ satisfies \eqref{cqd}.

\item[(4)] When hypothesis {\bf (H3)$'$} is verified, then the upper semicontinuity of the set valued map $\partial^* h(\cdot, \cdot)$ and \eqref{cqd} in  {\bf (H4)} imply that there exist  $\delta$, $\varepsilon >0$ such that for any $(\zeta_0, \zeta)\in \partial^* h(\sigma,x)$ with $\sigma \in [0,\varepsilon]$ and $x \in \{\xb\} + \varepsilon \B$, for any $s\leq \bar s$, for any continuous path $y:[0,s] \to \{\xb\}+\varepsilon\B$ and for any measurable map $\eta:[0,s] \to \{0,1\}$, the following integral condition holds:
\bel{cqd2}
\int_{0}^s \eta(\sigma) \,\zeta\cdot \big[\F(\sigma,y,\hat \w_i, \hat\alpha_i)(\sigma)-\F(\sigma,y, \tilde \w_i, \tilde\alpha_i)(\sigma)\big]d\sigma \leq - \delta \, \ell(s,\eta(\cdot)),
\eeq
where
\bel{ell}
\ell(s,\eta(\cdot)):= \ell(\{\sigma\in[0,s] \,:\, \eta(\sigma) =1\}).
\eeq
In particular,  relation \eqref{cqd2} holds  for any $(\zeta_0,\zeta)\in \partial^c h(\sigma,x)$, as the scalar product is bilinear, and   for all $\zeta\in \partial_x^> h(\sigma,x)$, since (see e.g. \cite[Th. 5.3.1]{OptV}):
$$
\partial_x^> h(\sigma,x) \subseteq \partial_x^c h(\sigma,x) \subseteq \{ \zeta \text{ : } \exists\zeta_0 \text{ s.t. } (\zeta_0,\zeta)\in \partial^c h(\sigma,x)\} \qquad \forall (\sigma,x)\in \R^{1+n}.
$$
Relation \eqref{cqd2} is in fact the condition  used  in the proof of Theorem \ref{Th2} below.

\item[(5)] When hypothesis {\bf (H3)$'$} is verified, it is not difficult to verify that condition \eqref{cqd2} still holds if we  replace {\bf (H4)}, (iv)   with the following assumption:
{\it\begin{itemize}
\item[{\rm (iv)$'$}] there exists $\tilde\varepsilon>0$ such that, for every $i$, for all $\zeta\in\partial_x^c h(\sigma,x)$ with $\sigma\in [0,\tilde\varepsilon]$ and $x \in \{\xb\} + \tilde\varepsilon \B$,   for a.e.  $s\in[0,\bar s]$ one has
  \bel{cqd3}
\begin{array}{l}
\ds \zeta\cdot  \big[\F(s,\xb,(\hat \w_i, \hat\alpha_i)(s))-\F(s,\xb, (\tilde \w_i, \tilde\alpha_i)(s))\big]\leq-\tilde\delta.
\end{array}
\eeq
\end{itemize}}
\end{itemize}}
\end{rem}

Theorem \ref{Th1}  can be refined as follows:  

\begin{theo}\label{Th2}  Let $\bar z:=(\underline{\bar\w},\underline{\bar\alpha},\bar\lambda,\bar\xi,\bar y)$ be a feasible relaxed process.  Assume that hypotheses  {\em {\bf (H1)}-{\bf (H2)}-{\bf (H3)$'$}-{\bf (H4)}} are verified. If   at $\bar z$ there is a local infimum gap, then $\bar z$ is a nondegenerate abnormal extremal. 
  \end{theo}

As in the previous section, from Theorem  \ref{Th2} one can derive the following results. 
\begin{theo}\label{Cor_Norm2}  
Suppose that there exists a  [local] $\Psi$-minimizer  $\bar z:=(\underline{\bar\w},\underline{\bar\alpha}, \bar\lambda,\bar\xi,\bar y)$  for  $(P_e)$ or  $(P_r)$, for which    {\em {\bf (H1)}-{\bf (H2)}-{\bf (H3)$'$}-{\bf (H4)}} are verified and $\Psi$ is Lipschitz continuous in a neighborhood of $\bar y(S)$. If  $\bar z$ is a nondegenerate normal $\Psi$-extremal, then, [at $\bar z$] there is no [local]  infimum gap with $(P_e)$ or  $(P_r)$, respectively.
\end{theo} 
 
\begin{theo}\label{Cor_Contr2}  Let $\bar z:=(\underline{\bar\w},\underline{\bar\alpha},\bar\lambda, \bar\xi,\bar y)$ be a feasible relaxed process and assume that  hypotheses   {\em {\bf (H1)}-{\bf (H2)}-{\bf (H3)$'$}-{\bf (H4)}} are verified. Then, either
\begin{itemize}
\item[{\rm (i)}] $\bar z$ is not isolated in $\Gamma$, namely, there exists a sequence of feasible processes $(\w_i,\alpha_i,y_i)\in\Gamma$ such that $\|y_i-\bar y\|_{L^\infty}\to 0$ as $i\to+\infty$; or
\item[{\rm (ii)}] $\bar z$ is a nondegenerate abnormal extremal.
  \end{itemize}
\end{theo}
The proof of Theorem \ref{Th2} is given in Section \ref{S4}.

\begin{rem} {\rm 
As it will be clear from the proofs in Section \ref{S4}, Theorems \ref{Th2}, \ref{Cor_Norm2}, \ref{Cor_Contr2} --as well as Theorems \ref{Th1}, \ref{Cor_Norm}, \ref{Cor_Contr}-- remain true if we replace the fixed set of control values $A$ with a compact Borel measurable multifunction $A:[0,S] \rightsquigarrow \R^q$.
}
\end{rem}

\section{Free end-time problems with Lipschitz time dependence}\label{SFT}
%
We  consider the optimization problem  
$$
(P^*)
\left\{
\parbox[c][2cm]{.9\textwidth}{%
\begin{align}
&\qquad\qquad\text{minimize} \,\,\,   \Psi(S,y(S)) \nonumber
 \\
 &\text{over the set of  $S>0$ and 
$(\w,\alpha,y)\in \V(S)\times\A(S) \times W^{1,1}([0,S];\R^n)$}, \nonumber \\
&\text{verifying the Cauchy problem \eqref{E} and the constraints} \nonumber \\
&h(s,y(s))\le0 \quad \forall s\in[0,S], \qquad 
(S,y(S))\in \T^*, \nonumber
\end{align}}
\right.
$$
where $\T^*$ is a closed subset of $\R^{1+n}$ and $\Psi:\R^{1+n}\to\R$. An    {\em extended process} or {\em process} is an element   $(S,\w,\alpha,y)$,  where $S>0$,  $(\w,\alpha)\in \W(S)\times\A(S)$,  and $y$ solves  the Cauchy problem \eqref{E}.
When  $\w\in\V(S)$, the process is called a {\em strict sense process}.  A process  $(S,\w,\alpha,y)$ is {\em feasible} if $h(s,y(s))\le0$   for all  $s\in[0,S]$ and 
$(S,y(S))\in\T^*$.  We  call {\em extended problem}, and write $(P^*_e)$, the  problem of minimizing $\Psi(S,y(S))$ over the set of feasible extended processes. 

The associated  {\em relaxed problem}  is
{\small
$$
(P^*_r)
\left\{
\parbox[c][2.7cm]{.9\textwidth}{%
\begin{align}
&\qquad\qquad\text{minimize} \,\,\,   \Psi(S,y(S)) \nonumber \\
&\text{over $S>0$, 
$(\underline{\w},\underline{\alpha},\lambda, y)\in \W^{1+n}(S)\times\A^{1+n}(S)\times  \Lambda_n(S)   \times W^{1,1}([0,S]; \R^n)$ s.t.}\nonumber \\
&\dot y(s)=\sum_{k=0}^n\lambda^k(s)\F(s,y(s),\w^k(s),\alpha^k(s)) \  \text{a.e. $s\in[0,S]$,} \quad
y(0)=\xb,\nonumber \\
&h(s,y(s))\le0 \quad \forall s\in[0,S], \qquad 
(S,y(S))\in\T^*, \nonumber
\end{align}}
\right.
$$
}
A process  $(S,\underline{\w},\underline{\alpha},\lambda, y)$ for $(P^*_r)$ is  referred to as {\em relaxed} process. As in the previous sections,  we can   identify the set  of strict sense processes  [extended processes]  with the subset  of relaxed processes   $(S,\underline{\w},\underline{\alpha},\lambda,y)$ with  $(\underline{\w},\underline{\alpha},\lambda)\in\V^{1+n}(S)\times\A^{1+n}(S)\times \Lambda_n^1(S)$  [$(\underline{\w},\underline{\alpha},\lambda)\in\W^{1+n}(S)\times\A^{1+n}(S)\times \Lambda_n^1(S)$]. We will use  $\Gamma^*$, $\Gamma^*_e$,   $\Gamma^*_{r}$ to denote the sets of  feasible strict sense, feasible extended,   and feasible relaxed processes, respectively. 
 
\vsm 
Throughout this section, we strengthen hypotheses {\bf (H2)}-{\bf (H3)} treating  time  as a state variable. As in Section \ref{S2}, we add to $(P^*_r)$  the variable $\xi(s)=\int_0^s\lambda(s')\,ds'$, $s\in[0,S]$   and call relaxed process any element $(S, \underline{\w},\underline{\alpha},\lambda,\xi,y)$,  where  $(\underline{\w},\underline{\alpha},\lambda)\in \W^{1+n}(S)\times\A^{1+n}(S)\times \Lambda_n(S)$ and $(\xi,y):=(\xi,y)[\underline{\w},\underline{\alpha},\lambda]$ on $[0,S]$.

\vsm We  shall consider  the following hypotheses, in which $(\bar S, \underline{\bar\w},\underline{\bar\alpha}, \bar\lambda,\bar\xi,\bar y)$ is  a given feasible relaxed process for $(P^*_r)$ and, for some $\theta>0$,   we set
$$
 \Sigma^*_{\theta}:=\left\{(t,x)\in\R\times\R^{n}: \ \ (t,x)\in (s, \bar y(s))+\theta\,\B, \ \ s\in[0,\bar S]\right\}.
$$  
{\em  \begin{itemize} 
\item[{\bf (H2)$^*$}]  The constraint function  $h$ is   $K_{_h}$-Lipschitz continuous  in $\Sigma^*_{\theta}$.
\item[{\bf (H3)$^*$}] 
{\rm (i)} The function $\F$ is continuous on $ \Sigma^*_{\theta}\times W\times A$.  Furthermore, there is some constant   $K_\F>0$ such that, for all $(s,x,w,a)$, $(s',x',w,a)\in \Sigma^*_{\theta}\times W\times A$:
\[
 |\F(s',x',w,a)-\F(s,x,w,a)|\le K_\F |(s',x')-(s,x)|.
\]
{\rm (ii)} There exists some continuous   increasing function $\varphi:\R_{\ge0}\to\R_{\ge0}$ with $\varphi(0)=0$ such that for any $(s,x,a)\in \Sigma^*_{\theta}\times A$,   we have
$$
\partial_{t,x}^c \F(s,x,w',a)\subseteq \partial_{t,x}^c \F(s,x,w,a)+\varphi(|w'-w|)\,\B \qquad\forall w',\,w\in W.
$$
 \end{itemize}}

Identify a continuous function $z:[0,\tau]\to\R^k$  with its extension to   $\R$, by constant extrapolation of  the left and right endpoint values. Then, for all   $\tau_1$,   $\tau_2>0$,   and    $(z_1, z_2)\in C^0([0,\tau_1];\R^k)\times C^0([0,\tau_2];\R^k)$, we define the distance
 \begin{equation}\label{dinfty}
d_\infty\big((\tau_1,  z_1),(\tau_2, z_2)\big) :=  
  | \tau_2 -  \tau_1|+ \| z_2-  z_1\|_\infty \qquad ( \| \cdot\|_{\infty}:=  \| \cdot\|_{L^\infty(\R)}).
  \end{equation} 
We can now extend the concepts of {\em local minimizer}, {\em local infimum gap}, and {\em isolated process}  to free end-time problems,  by `formally replacing  trajectories $y$ and $L^\infty$-norm over trajectories  with pairs $(S,y)$ endowed with the distance $d_\infty$'. For instance, if  $\bar z:=(\bar S,\underline{\bar\w},\underline{\bar\alpha},\bar\lambda, \bar y)$ is a feasible relaxed process,  {\em at  $\bar z$ there is  a local infimum} gap if there is some $\delta>0$ such that 
$$
 \ds \Psi(\bar S,\bar y(S))  <\,  \inf \left\{ \Psi(S,y(S)): \ \ (S,\w,\alpha,y)\in\Gamma^*, \ \ d_\infty\big((S,y),(\bar S,\bar y)\big)<\delta \right\}, 
 $$
for some continuous function  $\Psi:\R^{1+n}\to\R$,  while $\bar z$ is an {\em isolated process} if  
$$
\left\{ (S,\w,\alpha,y)\in  \Gamma^*:  \ \ d_\infty\big((S,y),(\bar S,\bar y)\big)<\delta \right\}=\emptyset.
$$
for some $\delta>0$. As in the case with fixed end-time, {\em at $\bar z$ there is a   local infimum  gap if and only if   $\bar z$ is isolated}.

\begin{definition}[Extremal and nondegenerate extremal] \label{DnormalFT}	 
Let  $\bar z:=(\bar S,\underline{\bar\w},\underline{\bar\alpha}, \bar\lambda,\bar y)$ be a feasible relaxed  process and assume that hypotheses {\bf (H2)$^*$}, {\bf (H3)$^*$} are verified.  Given a   function $\Psi:\R^{1+n}\to\R$ which is Lipschitz continuous on a neighborhood of $(\bar S,\bar y(\bar S))$,      we say that $\bar z$  is  a  {\em  $\Psi$-extremal}  if there exist a pair of paths $(p_*,p) \in W^{1,1}([0, \bar S]; \R^{1+n})$, $\gamma \geq 0$,  $\mu \in NBV^{+}([0, \bar S];\R)$, $(m_*,m): [0,\bar S] \to \R^{1+n}$ Borel measurable and $\mu$-integrable functions, verifying the following conditions: 
$$
\parbox[c][6.2cm]{1\textwidth}{%
\begin{align}
&\| p \|_{L^{\infty}}+ \| \mu \|_{TV} + \gamma \neq 0; \label{nontr} \\
& \big( \dot p_*, -\dot p\big)(s)  \in \sum_{k=0}^n \bar\lambda^k(s)\, {\rm co}\,  \partial_{t,x}  \left\{q(s)\cdot \F(s, \bar y(s),  \bar \w^k(s),\bar\alpha^k(s))\right\} \  \text{a.e. $s \in [0,\bar S]$;}  \nonumber \\
& \big(q_*(\bar S), -q(\bar S) \big) \in \gamma \partial \Psi\left( S, \bar y(S)\right)  +N_{\T } (\bar S, \bar y(\bar S) ); \nonumber \\
&\text{for every $k=0,\dots,n$,  for a.e. $s\in[0,\bar S]$, one has } \nonumber \\
&q(s)\cdot \F\Big(s, \bar y(s), \bar \w^k(s),\bar\alpha^k(s)\Big) 
 =  \max_{(w,a)\in W\times A}  q(s)\cdot\F\Big(s, \bar y(s), w,a\Big); 
 \label{cmax} \\
 & \sum_{k=0}^n \bar\lambda^k(s) q(s)\cdot \F(s, \bar y(s),  \bar \w^k(s),\bar\alpha^k(s))= q_*(s)  \ \text{a.e. $s \in [0,\bar S]$}; \label{cHam} \\
 &(m_*,m) (s) \in \partial_{t,x}^{>}\, h\left(s, \bar y(s)\right)   \text{ $\mu$-a.e. $s \in  [0,\bar S]$;}  \nonumber\\
& spt(\mu) \subseteq \{ s\in [0,\bar S] \text{ : } h\left(s, \bar y(s)\right) = 0 \}, \nonumber
\end{align}}
$$
where
$
(q_*,q)(s)  := 
\begin{cases}
(p_*,p)(s) + \int_{[0,s[} (m_*,m)(\sigma) \mu(d\sigma)  \qquad\,\, s\in [0,\bar S[, \\
(p_*,p)(\bar S) + \int_{[0, \bar S]} (m_*,m)(\sigma) \mu(d\sigma)  \qquad s=\bar S.
\end{cases}
$
%
%
%
%

\noindent A  $\Psi$-extremal  is {\em normal} if all possible choices of   $(p_*,p,\gamma,\mu, m_*,m)$ as above have $\gamma >0$, and   {\em abnormal}  when it is not normal. 
Given  a   $\Psi$-extremal   $\bar z$, we call {\em nondegenerate multiplier} any  set of multipliers  $(p_*,p,\gamma,\mu, m_*,m)$  and $(q_*,q)$ as above,   that also verify
\begin{equation} \label{nondegFT}
\mu(]0,S]) +\Vert q \Vert_{L^{\infty}} + \gamma \neq 0.
\end{equation}
A  $\Psi$-extremal is  {\em nondegenerate normal} if all the choices of  nondegenerate multipliers  have $\gamma >0$, and  it is {\em nondegenerate abnormal}  when there exists a  nondegenerate multiplier with $\gamma=0$.  In the following, abnormal [nondegenerate abnormal]  $\Psi$-extremals will be simply called   {\em abnormal  [nondegenerate abnormal] extremals}.  
 \end{definition}  
 
 Theorems \ref{Th1}, \ref{Th2} extend to free end-time optimization problems as follows.
\begin{theo}\label{Th12FT}  Let $\bar z:=(\bar S, \underline{\bar\w},\underline{\bar\alpha},\bar\lambda,\bar\xi,\bar y)$ be a feasible relaxed process for $(P_r^*)$, and suppose that at $\bar z$ there is a local infimum gap. 
\begin{itemize}
\item[{\rm (i)}] If hypotheses  {\bf (H1)}-{\bf (H2)$^*$}-{\bf (H3)$^*$} hold, then $\bar z$ is an abnormal extremal. 
\item[{\rm (ii)}] If, in addition, also hypothesis  {\bf (H4)} for $S=\bar S$ is verified, then $\bar z$ is a nondegenerate abnormal extremal.
\end{itemize}
  \end{theo}

\begin{proof}
 Let $\bar z:=(\bar S, \underline{\bar\w},\underline{\bar\alpha},\bar\lambda,\bar\xi,\bar y)$ be a feasible relaxed process at which there is a local infimum gap. By the above considerations, this is equivalent to suppose that $\bar z$ is an isolated process. Adapting a  standard time-rescaling procedure (see e.g. \cite[Thm. 8.7.1]{OptV}), we transform    problems $(P^*)$, $(P_e^*)$, and $(P_r^*)$  into   fixed end-time problems,  and show that $\bar z$ is  an isolated process also with respect to feasible strict sense processes  of a {\em rescaled, fixed end-time problem}. At this point, the thesis follows by  applying Theorems \ref{Th1}, \ref{Th2} to the rescaled problem.

From the fact that $\bar z$ is isolated, it follows that there exists some  $\delta>0$ such that
\bel{isoHp}
\left\{ (S,\w,\alpha,y)\in  \Gamma^*:  \ \ d_\infty\big((S,y),(\bar S,\bar y)\big)<3\delta \right\}=\emptyset.
\eeq
Set
 $\bar\delta:=\min\left\{\frac{\delta}{3\bar S\,K_\F},\frac{1}{2}\right\}$.
 
 We define  the  {\it rescaled  optimization problem} $(\hat P^*)$,  as
{\small
$$
(\hat P^*) 
\left\{
\parbox[c][2cm]{0.9\textwidth}{%
\begin{align}
&\qquad\qquad\text{minimize} \,\,\,  \Psi(y^*(\bar S),y(\bar S)) \nonumber \\
&\text{over   $(\w,\alpha,\zeta, y^*, y)\in \V (\bar S)\times\A (\bar S) \times L^{ 1}([0,\bar S];[-\bar\delta,\bar\delta])\times W^{1,1}([0,\bar S]; \R^{1+n})$ s.t. }\nonumber \\
& (\dot y^*, \dot y)(s) = (1+\zeta(s)) \big( 1,  \F(y^*(s),y(s),\w(s),\alpha(s)) \big)  \ \text{a.e.}, \quad (y^*,y)(0)=(0,\xb)   \nonumber \\
&h(y^*(s),y(s))\le0 \quad \forall s\in[0,\bar S], \qquad  
(y^*(\bar S),y(\bar S))\in\T^*, \nonumber
\end{align}}
\right.
$$
}
%
A process  $(\w,\alpha,\zeta, y^*, y)$ for the  fixed end-time problem $(\hat P^*)$ is referred to as a {\em rescaled strict sense process}. Let $\hat\Gamma^*$  denote the set of the {\em feasible} rescaled strict sense processes, that is, the set  rescaled strict sense processes that verify $h(y^*(s),y(s))\le0$   for all $s\in[0,\bar S]$ and 
$(y^*(S),y(S))\in\T^*.$  We call {\em rescaled extended processes} the processes $(\w,\alpha,\zeta, y^*, y)$ with $\w\in \W(\bar S)$ and  write $(\hat P_e^*)$ to denote the rescaled  extended problem associated to $(\hat P^*)$.  

\vsm
We can identify   $\bar z=(\bar S, \underline{\bar\w},\underline{\bar\alpha},\bar\lambda,\bar\xi,\bar y)$  with a process  $
 \check z:=(\underline{\check\w},\underline{\check\alpha},\underline{\check\zeta},\check\lambda,\check \xi,  \check y^*, \check y)\in\W^{2+n}(\bar S)\times \A^{2+n}(\bar S)\times(L^1[0,\bar S];[-\bar\delta,\bar\delta]))^{2+n}\times\Lambda_{n+1}(\bar S)\times W^{1,1}([0,\bar S];\R^{2+n}\times\R^{1+n})
$
of the relaxed problem associated to $(\hat P^*_e)$, by setting  
$$
\begin{array}{l}
\underline{\check\w}:=(w,\underline{\bar\w}), \ \ \underline{\check\alpha}:=(a,\underline{\bar\alpha}), \ \ \underline{\check\zeta}:=0, \ 
 \check\lambda:= (0,\bar\lambda), \ \ \check\xi=(0,\bar\xi),  \  \ \check y^*:=id, \ \ \check y := \bar y,  
\end{array}
$$
for   arbitrary $w\in W$ and $a\in A$.  Since  $\bar z$ is an isolated feasible relaxed process for the free end-time problem, $\check z$ is feasible and  isolated  for the relaxed rescaled problem. In particular, we claim that
\bel{isores}
\left\{ (\w,\alpha,\zeta,y^*,y) \in  \hat\Gamma^*:  \ \ \|(y^*,y)-(\check y^*,\check y) \|_{L^\infty([0,\bar S])}< \delta \right\}=\emptyset.
\eeq
Indeed, let   $(\w,\alpha,\zeta,y^*,y)$ be an arbitrary feasible, rescaled strict sense process verifying \linebreak$\|(y^*,y)-(\check y^*,\check y)\|_{L^\infty([0,\bar S])}< \delta $. Consider the time-transformation $y^*:[0,\bar S]\to[0,S]$, where 
$
 S:=y^*(\bar S).
 $
 Observe that $y^*$ is a strictly increasing, Lipschitz continuous function, with Lipschitz continuous inverse, $(y^*)^{-1}$. 
It can be deduced that the process $(S,\hat\w,\hat\alpha,\hat y)$, where
$$
(\hat\w,\hat\alpha,\hat y):=(\w,\alpha,y)\circ  (y^*)^{-1}  \quad \text{in $[0,S]$,}
$$  
is a feasible  strict sense process for the  free end-time problem  $(P^*)$,  i.e. $(S,\hat\w,\hat\alpha,\hat y)\in\Gamma^*$. Indeed, recalling the definitions  of $\bar\delta$ and $d_\infty$,\footnote{By definition, $\hat y$ and $\check y$ are replaced with their constant, continuous extensions  to $\R$.}   after some calculations, we get
$$ 
\begin{array}{l}
\ds d_\infty\big((S,\hat y),(\bar S,\bar y)\big)= 
 |S-\bar S|+\|\hat y-\bar y\|_{\infty} \le |y^*(\bar S)-\check y^*(\bar S)| \\[1.5ex]
\ds \  + \sup_{s\in[0, S\vee\bar S]}\left[|y((y^*)^{-1}(s\land S))-\check y((y^*)^{-1}(s\land  S))| +|\check y((y^*)^{-1}(s\land S))-\check y(s\land \bar S)|\right] \\
 \qquad \qquad  \le\|y^*-\check y^*\|_{L^\infty([0,\bar S])}+ \|y-\check y\|_{L^\infty([0,\bar S])} + \delta<  3\delta,
\end{array}
$$
since, in particular,   
$$
\sup_{s\in[0,S\vee\bar S]}|\check y((y^*)^{-1}(s\land S))-\check y(s\land \bar S)|\le  3K_\F  \bar S\,\bar\delta \le \delta .
$$ 
Therefore,   \eqref{isoHp} yields  
 \eqref{isores}, and   the feasible rescaled  relaxed process
$\check z$
 is isolated in $\hat\Gamma^*$, as claimed. 
In order to apply the results of Theorems \ref{Th1}, \ref{Th2} to $(\hat P^*)$  with reference to the process $\check z$, 
it remains to show that,  if we consider $(w,a,\lambda,\zeta)$ as control variables and $\tilde x:=(s,x)$ as  the state variable for the (now, time-independent)  problem $(\hat P^*)$,  all the hypotheses assumed in their statements are fulfilled.  To this aim, observe that  {\bf (H2)} trivially follows from {\bf (H2)$^*$},  while  {\bf (H3)$^*$} easily  implies  {\bf (H3)$'$}. In particular,  the compactness of  $\Sigma^*_\theta\times W\times A\times[-\bar\delta,\bar\delta]$ and the continuity of $(1+\zeta)\F$ on it,  guarantee the existence of a constant $M_\F>0$ and of a continuous increasing function $\varphi:\R_{\ge0}\to \R_{\ge0}$ with $\varphi(0)=0$ such that
$$
|(1+\zeta)\F(t,x,w,a)|\le M_\F , \quad |(1+\zeta)\F(t,x,w',a)-(1+\zeta)\F(t,x,w,a)|\le\varphi(|w'-w|),
$$
 for all  $(t,x,w',a,\zeta)$, $(t,x,w,a,\zeta)\in\Sigma^*_\theta\times W\times A\times[-\bar\delta,\bar\delta]$. Finally,  recalling that $\underline{\check\zeta}\equiv 0$ and   $\check y^*(s)=s$ for all $s\in[0,\bar S]$,  hypothesis 
{\bf (H4)} can be trivially reformulated as an hypothesis on the rescaled process $\check z$.
At this point, from Theorems \ref{Th1}, \ref{Th2}   we can derive that $\bar z$ is an abnormal extremal, nondegenerate  when hypothesis {\bf (H4)} is verified. In particular,  the only nontrivial results, namely conditions \eqref{cmax}, \eqref{cHam}, and the nontriviality conditions \eqref{nontr}, \eqref{nondeg} can be obtained  through routine arguments  (see e.g. the proof of \cite[Thm. 8.2.1]{OptV}). 
\end{proof}

Again, from Theorem  \ref{Th12FT} we can get normality tests for gap avoidance and sufficient controllability conditions 
 for the free end-time problem completely analogous to Theorems  \ref{Cor_Norm}- \ref{Cor_Contr} and  \ref{Cor_Norm2}-\ref{Cor_Contr2}, respectively.  
 

\section{An application to   non-convex, control-polynomial impulsive  problems}\label{S3}
We consider the free end-time optimal control problem: 
{\small 
$$
(\mathcal{P}^*)
\left\{
\parbox[c][3.1cm]{0.9\textwidth}{%
\begin{align}
&\qquad\qquad\text{minimize} \,\,\,  \Psi(T,x(T),v(T))  \nonumber
 \\
&  \text{over }  T>0,  \, (u,a,x,v) \in  L^d([0,T];U) \times L^1([0,T];A)\times  W^{1,1}([0,T];\R^{n+1})\ \text{ s.t.} \nonumber \\
&(\dot x, \dot v)(t)= \Bigg(f(t,x,a)+\sum_{k=1}^d\Bigg(\sum_{1\le j_1\leq\dots\leq j_k\le m}   g^k_{j_1,\dots,j_k}(t,x)\, \,u^{j_1}\cdots u^{j_k}\Bigg),   \left|u\right|^d \Bigg)\,  \text{a.e.,}  \nonumber \\ 
& (x,v)(0)=(\xb,0)  \nonumber \\
&h(t, x(t)) \leq 0 \qquad \text{ for all } t \in [0,T], \quad v(T)\leq K,     \quad (T,x(T)) \in \T^*.   \nonumber
\end{align}}
\right.
$$
}
%
%
Here, $U\subseteq\R^m$ is a  closed cone, $A\subseteq\R^q$ is a compact subset, $K>0$ is a  fixed constant, possibly equal to $+\infty$, and the target set $\T^*\subseteq\R^{1+n}$ is closed.  Notice   that $v(t)$ is simply  the  $L^d$-norm to the power $d$  of the control function $u$  on $[0,t]$. The variable $v$ is sometimes called   {\em fuel} or {\em energy}   and   $v\mapsto \Psi(t,x,\cdot)$  is usually assumed monotone nondecreasing for every $(t,x)$ (see e.g. \cite{MS03,MS14}).       The integer $d\ge1$ will be   called  the {\em degree}  of the control system.
Problem $(\mathcal{P}^*)$  is referred to as the {\em original problem} and  we call a process $(T,u,a,x,v)$ for $(\mathcal{P}^*)$ an {\em original process}.  
 We say that  $(T,u,a,x,v)$ is  {\em feasible} if $h(t, x(t)) \leq 0$   for all $t\in [0,T]$,  $v(T)\leq K$,    and $(T,x(T)) \in \T^*$.  

 \vsm
Throughout this section,  we shall consider the following structural  hypotheses: 
{\em\begin{itemize}
\item[{\bf (H5)}] the  functions $f:\R^{1+n}\times A\to \R^n$, $g^k_{j_1\dots j_k}:\R^{1+n}\to\R^n$ are continuous, all  $g^k_{j_1\dots j_k}$ are locally Lipschitz continuous, and $f(\cdot,\cdot,a)$  is locally Lipschitz continuous uniformly w.r.t. $a\in A$.   Furthermore, the constraint function $h:\R^{1+n}\to\R$ is  locally Lipschitz continuous.
\end{itemize}}

In order to apply the theory developed in the previous sections,   we reformulate problem $(\mathcal{P}^*)$  and embed it into a  free end-time extended problem with bounded  controls.  To do this, we use a compactification procedure based on a reparameterization technique,  commonly  adopted to obtain an impulsive extension of  unbounded control problems, as generalized to polynomial systems (see e.g. \cite{RS00,MS14}).  Let us choose
$$
W:=\left\{(w^0,w)\in\R_{\ge0}\times U: \ \ (w^0)^d+|w|^d=1\right\}, \quad V:=\left\{(w^0,w)\in W: \ \ w^0>0\right\}.
$$
For every $S>0$,  we set 
$\W(S):= L^1([0,S]; W)$,\footnote{The controls  $(\w^0,\w)\in\W(S)$ actually belong to $L^\infty\cap L^1$, since $W$ is compact.} $\V(S):= L^1([0,S]; V)$, and $\A(S):= L^1([0,S]; A)$, and 
 introduce the {\it space-time} or {\it extended  problem}: \footnote{The original time $t$ coincides now with the state component $y^0$, while $s$ is  the new `pseudo-time' variable.}
$$
(P^*_e)
\left\{
\parbox[c][2.8cm]{0.9\textwidth}{%
\begin{align}
&\qquad\qquad\text{minimize} \,\,\,  \Psi( y^0(S), y(S),\nu(S))  \nonumber \\
&\text{over $S>0$, }  (\w^0,\w,\alpha, y^0, y,\nu)  \in \W(S)\times\A(S)\times W^{1,1}([0,S];\R^{1+n+1})  \text{ s.t.} \nonumber \\
&(\dot y^0, \dot y, \dot\nu)(s)= \Big((\w^0)^d(s) , \F(y^0(s),y(s),\w^0(s),\w(s),\alpha(s)),    \left|\w(s)\right|^d \Big) \ \text{a.e.,}  \nonumber \\ 
& (y^0,y,\nu)(0)=(0,\xb,0),  \nonumber \\
&h(y^0(s),y(s)) \leq 0 \ \text{ for all } s \in [0,S] , \quad 
\left(y^0(S),y(S),\nu(S)\right)\in \T^*\times]-\infty,K],   \nonumber
\end{align}}
\right.
$$
%
%
%
%
where, for any $(t,x,w^0,w,a)\in\R\times\R^{n}\times\R_{\ge0}\times U\times A$, we have set
{\small
$$
\F(t,x,w^0,w,a):=f(t,x,a)(w^0)^d + 
 \sum_{k=1}^d\Bigg(\sum_{1\le j_1\leq\dots\leq j_k\le m}   g^k_{j_1,\dots,j_k}(t,x)\, \,w^{j_1}\cdots w^{j_k}\,(w^0)^{d-k}\Bigg).
$$
}
Adopting notation and  terminology of Section \ref{SFT}, a process  $(S,\w^0,\w,\alpha, y^0, y,\nu)$ of problem $(P^*_e)$ is referred to as an {\em extended} process  and it is {\em feasible} if $h(y^0(s),y(s)) \leq 0$ for all  $s \in [0,S]$  and
$\left(y^0(S),y(S),\nu(S)\right)\in \T^*\times]-\infty,K]$. When $\w^0>0$ almost everywhere, namely $(\w^0,\w)\in\V(S)$, $(S,\w^0,\w,\alpha, y^0, y,\nu)$  is said a {\em strict sense} process. The problem of minimizing $\Psi(y^0(S),y(S))$ over feasible strict sense processes is still denoted  by $(P^*)$,   and the sets of feasible extended processes and feasible strict sense processes are $\Gamma$ and $\Gamma^*$, respectively.   

 \vsm 
The original problem $(\mathcal{P}^*)$ can be identified with problem $(P^*)$, as established by the following lemma, immediate consequence of the chain rule.
 \begin{lemma}[Embedding]\label{EmbIMP} Assume hypothesis {\bf (H5)}. Then 
 the map  
\[
 {\mathcal I}: \{(T, u,a, x ,v),    \ \text{original processes}\}\to\{(S,\w^0,\w,\alpha, y^0, y,\nu),   \ \text{extended processes}\}
\]
defined as
$$
 {\mathcal I}(T, u,a, x ,v):= (S,\w^0,\w,\alpha,\zeta,y^0,y,\nu),
$$
where,  setting $\sigma(t):=t+v(t)$  for all $t\in[0,T]$,  
\[
\begin{array}{l}
\ds S:=\sigma(T), \quad (y^0,y,\nu)(s):=(id,x,v)\circ \sigma^{-1}(s) \ \ \forall s\in[0,S], \\ 
\ds(\w^0,\w)(s):=(1+|u|^d)^{-\frac{1}{d}} \, \,(1,u)\circ \sigma^{-1}(s), \quad \alpha(s):=a\circ \sigma^{-1}(s) \ \ \text{a.e. $s\in[0,S]$,}\footnotemark
\end{array}
\]
\footnotetext{Since  every   $L^d$-equivalence class contains  Borel measurable representatives,   we are tacitly assuming that all $L^d$-maps we are considering are Borel measurable.}
is injective and has as image the subset of strict sense processes.  Moreover,  $ {\mathcal I}$ maps any {\em feasible} original  process into a {\em feasible} strict sense process, with the same cost.
 \end{lemma}
 The extended problem $(P^*_e)$ consists thus  in considering processes $(S,\w^0,\w,\alpha, y^0, y,\nu)$,  where $\w^0$ may be zero on   nondegenerate subintervals of $[0,S]$. On these intervals, the time variable $t=y^0$ is constant  --i.e. the time  stops--, while the state variable $y$ evolves, according to $\F(t,y,0,w,\alpha)=$\linebreak $\sum_{1\le j_1\leq\dots\leq j_d\le m}   g^d_{j_1,\dots,j_d}(t,y)\, \,w^{j_1}\cdots w^{j_d}$, which can be called {\em fast dynamics}. For this reason,   problem $(P^*_e)$ is  often referred to  as the {\em impulsive extension} of the original problem $(\mathcal{P}^*)$  (more details on  polynomial impulsive problems can be found  in  \cite{RS00,MS14} and  references therein). 
 
\vsm

Let us introduce the   {\it unmaximized Hamiltonian} $H$, defined  by
$$
\begin{array}{l}
H(t,x,p_0,p,\pi,\w^0,\w,a)   := p_0(\w^0)^d + p\cdot\F(t,x,w^0,w,a) + \pi | \w|^d,
\end{array}
$$ 
for all $(t,x,p_0,p,\pi,\w^0,\w,a)\in \R^{1+n+1+n+1 }\times W\times A$. 
The concepts of {\em extremal}  and {\em nondegenerate extremal}  read now as follows:  
\begin{definition}[Extremal and nondegenerate extremal]\label{extremalImp}
Assume  {\bf (H5)} and let $\bar z:=(\bar S, \bar\w^0,\bar\w,\bar\alpha, \bar{y}^0,\bar y,\bar\nu)$ be a feasible extended  process.  Given a cost  function $\Psi$ which is Lipschitz continuous on a neighborhood of $(\bar y^0(\bar S),\bar y(S),\bar\nu(\bar S))$,      we say that $\bar z$  is  a  {\em  $\Psi$-extremal} if there exist a path $(p_0,p) \in W^{1,1}([0, \bar S];\R\times\R^n)$, $\gamma \geq 0$, $\pi \leq 0$, $\mu \in NBV^{+}([0, \bar S];\R)$, $(m_0,m) : [0, \bar S] \to \R^{1+n}$ Borel measurable and $\mu$-integrable functions, verifying the following conditions:
$$
\parbox[c][6.2cm]{1\textwidth}{%
\begin{align}
&\| p_0 \|_{L^{\infty}} +\| p \|_{L^{\infty}}+ \| \mu \|_{TV} + \gamma \neq 0;  \label{fe1imp} \\
&(-\dot p_0, -\dot p)(s) \in {\rm co} \ \partial_{t,x} \ H\Big(\bar y^0(s), \bar y(s), q_0(s) , q(s),\pi,  \bar \w^0(s),\bar \w(s),\bar\alpha(s)\Big)\ \text{a.e.;}  \nonumber \\
&\left(-q_0(\bar S), -q(\bar S) , -\pi \right) \in \gamma \partial \Psi\left(\bar{y}^0(\bar S), \bar y(\bar S), \bar\nu(\bar S)\right)  \nonumber \\
&\qquad\qquad\qquad\qquad \qquad\qquad  +N_{\T^*\times]-\infty,K]}\left(\bar{y}^0(\bar S), \bar y(\bar S),   \bar\nu(\bar S)\right);  \label{trans_condimp} \\
&\text{ for a.e. $s\in[0, S]$, one has } \nonumber \\
&H\Big(\bar y^0(s), \bar y(s), q_0(s) , q(s),\pi,  \bar \w^0(s),\bar \w(s),\bar\alpha(s)\Big) \nonumber \\
& \qquad\qquad=  \max_{(\w^0,\w,a)\in W\times A}  H\Big(\bar y^0(s), \bar y(s), q_0(s) , q(s),\pi,  \w^0,\w,a\Big)=0;
 \label{maxhamimp} \\
 &(m_0,m)(s) \in \partial_{t,x}^{>}\, h\left(\bar{y}^0(s), \bar y(s)\right) \qquad  \text{$\mu$-a.e.;} \nonumber \\
& spt(\mu) \subseteq \{ s\in[0,\bar S] \text{ : } h\left( \bar{y}^0(s), \bar y(s)\right) = 0 \}, \nonumber
\end{align}}
$$
%
%
%
where
$
(q_0,q)(s) := 
\begin{cases}
 (p_0,p)(s) +  \int_{[0,s[} (m_0,m)(\tau) \mu(d\tau)\, \qquad \qquad \qquad \,\, s\in[0, \bar S[, \\
  (p_0,p)(\bar S) + \int_{[0, \bar S]} (m_0,m)(\tau) \mu(d\tau)\, \qquad \qquad\qquad s=\bar S.
\end{cases}
$

Given a $\Psi$-extremal  $\bar z$ we call {\em nondegenerate multipliers}  all   $(p_0,p,\pi,\gamma,\mu, m_0,m)$  and $(q_0,q)$ as above,   that also verify
\begin{equation} \label{nondegFTimp}
\mu(]0,S]) +\Vert q_0 \Vert_{L^{\infty}}  +\Vert q \Vert_{L^{\infty}} + \gamma \neq 0.
\end{equation}

\noindent If $\gamma \partial_v\Psi\left(\bar{y}^0(\bar S), \bar y(\bar S), \bar\nu(\bar S)\right) = 0$ and $\bar\nu(\bar S) < K$, then $\pi=0$. Furthermore,  if $\bar y^0(0) < \bar y^0(\bar S)$,   \eqref{fe1imp} [\eqref{nondegFTimp}] can be strengthened to 
$$
\| p \|_{L^{\infty}}+ \| \mu \|_{TV} + \gamma \neq 0 \quad  [\mu(]0,S]) +\Vert q \Vert_{L^{\infty}} + \gamma \neq 0].
$$
\end{definition} 

In  the special  case of the impulsive extension considered in this section, given a  feasible extended  process $(\bar S, \bar\w^0,\bar\w,\bar\alpha, \bar{y}^0,\bar y,\bar\nu)$ the constraint qualification condition {\bf (H4)} for nondegeneracy  can be replaced by the following weaker assumptions:\footnote{We recall that  $\Omega=\{(t,x): \ h(t,x)\le 0\}$.}
 \vsm
{\em  \begin{itemize} 
\item[{\bf (H6)}] If $(0,\xb)\in \partial \Omega$,  there are some   $\tilde\delta>0$,  $\bar s\in]0,\bar S]$,  some   sequence of  strict sense processes \linebreak $(\tilde \w^0_i, \tilde \w_i,\tilde\alpha_i, \tilde y^0_i,\tilde y_i,\tilde\nu_i)_i\subset \V(\bar s)\times\A(\bar s)\times W^{1,1}([0,\bar s];\R^{1+n+1})$, some sequences  $(\hat \w_i^0,\hat \w_i,\hat\alpha_i)_i\subset \W(\bar s)\times\A(\bar s)$,  and   $(\tilde r_i)_i\subset L^1([0,\bar s];\R_{\ge0})$ with $\ds\lim_{i\to+\infty}\|\tilde r_i\|_{L^1([0,\bar s])}=0$,  such that the following properties {\rm (i)}--{\rm (iii)} are verified.
\begin{itemize}
\item[{\rm (i)}]
 For every $i$, one has
$$
  h(\tilde y^0_i(s),\tilde y_i(s))\le 0 \qquad \forall s\in[0,\bar s]; 
  $$
 \item[{\rm (ii)}]
 for every $i$, there is a Lebesgue measurable subset $\tilde E_i\subset[0,\bar s]$ such that  
$$
 \begin{array}{c}
 (\tilde \w^0_i, \tilde \w_i)(s)\in (\bar \w^0,\bar\w)(s) + \tilde r_i(s) \B,  \quad  \tilde\alpha_i(s)=\bar\alpha(s), \quad \text{a.e. $s\in \tilde E_i$;} \\[1.5ex]
  \lim_{i\to+\infty}\ell(\tilde E_i)=\bar s;
  \end{array}
  $$
  \item[{\rm (iii)}] for every $i$,  
  for all $(\zeta_0, \zeta)\in\partial^* h(0, \xb)$,  and for a.e.  $s\in[0,\bar s]$,  one has
 $$
\begin{array}{l}
\ds  \zeta_0\cdot [(\hat \w_i^0(s))^d-(\tilde \w_i^0(s))^d] \\
 \qquad+ \zeta\cdot  \big[\F(0 ,\xb,(\hat \w_i^0,\hat \w_i, \hat\alpha_i)(s)) -\F( 0,\xb, (\tilde \w_i^0, \tilde \w_i, \tilde\alpha_i)(s))\big] \leq-\tilde\delta.
\end{array}
$$
\end{itemize}
\end{itemize}}
\vsm
  In some situations, hypothesis {\bf (H6)} simplifies considerably. 
  \begin{lemma}\label{LH6} Assume {\bf (H5)}. Let $(0,\xb)\in\partial\Omega$   and let $\bar z:=(\bar S, \bar\w^0,\bar\w,\bar\alpha, \bar{y}^0,\bar y,\bar\nu)$ be a feasible  extended process.   If   there are some   $\tilde\delta>0$,  $\bar s\in]0,\bar S]$, and an extend control  $(\hat \w^0,\hat \w,\hat\alpha)\in \W(\bar s)\times\A(\bar s)$ such that,   for all $(\zeta_0, \zeta)\in\partial^* h(0, \xb)$ and for a.e.  $s\in[0,\bar s]$, 
\bel{cqdsimple}
\begin{array}{l}
\ds  \zeta_0\cdot [(\hat \w^0(s))^d-(\bar\w^0(s))^d] \\
\ \ + \zeta\cdot  \big[\F(0,\xb,(\hat \w^0,\hat \w, \hat\alpha)(s)) -\F(0,\xb, (\bar\w^0, \bar\w, \bar\alpha)(s))\big] \leq-\tilde\delta.
\end{array}
\eeq
and either  $\bar\w^0>0$   a.e. in $[0,\bar s]$, or  there is some $\tilde\delta_1>0$ such that,    for a.e.  $s\in[0,\bar s]$,
\bel{ipfc}
\begin{array}{l}
\ds \sup_{(\zeta_0, \zeta)\in\partial^* h(0, \xb)}\left[ \zeta_0\cdot  (\bar\w^0(s))^d 
 + \zeta\cdot  \F(0,\xb, ((\bar\w^0, \bar\w), \bar\alpha)(s))\right]  \leq-\tilde\delta_1,  
 \end{array}
\eeq
then condition  {\bf (H6)} is satisfied.
 \end{lemma}
 \begin{proof}[Proof of Lemma \ref{LH6}]  Let us first suppose that  $\bar\w^0>0$   a.e. in $[0,\bar s]$. Then,  conditions {\bf (H6)},(i),(ii)  are verified by  choosing, for every $i$,  $(\tilde \w^0_i, \tilde \w_i,\tilde\alpha_i)=(\bar w^0,\bar w,\bar\alpha)$, while  {\bf (H6)},(iii) follows directly from \eqref{cqdsimple}, by taking $(\hat \w_i^0,\hat \w_i, \hat\alpha_i)\equiv(\hat \w^0,\hat \w,\hat\alpha)$ for every $i$.  If instead \eqref{ipfc} is  assumed, let us consider a sequence $\delta_i\downarrow0$  and for every $i$, let us  set
 $$
( \tilde \w^0_i,\tilde \w_i,\tilde\alpha_i)(s):=\begin{cases} (\bar \w^0,\bar \w,\bar\alpha)(s)\  \  \ \qquad\qquad \text{if $\bar \w^0(s)>0$,} \\
 (\delta_i,\sqrt[d]{1-\delta_i^d}\,\bar \w(s),\bar\alpha(s)) \  \ \text{if $\bar \w^0(s)=0$,}
\end{cases}
$$
 for a.e.  $s\in[0,\bar s]$.  Then,   $( \tilde \w_i^0,\tilde \w_i,\tilde\alpha_i)\in\V(\bar s)\times\A(\bar s)$,  it verifies {\bf (H6)},(ii)  with $\tilde E_i=[0,\bar s]$   and  $\tilde r_i\equiv\delta_i$. Let $(\tilde y_i^0,\tilde y_i,\tilde\nu_i)$ be the corresponding solution   of the extended control system in $(P^*_e)$ with initial condition $(\tilde y^0_i,\tilde y_i,\tilde\nu_i)(0)=(0,\xb,0)$.
From   condition \eqref{ipfc}, using the  Lebourgh Mean Value Theorem   \cite[Th. 4.5.3]{OptV}    to estimate  `$h(\tilde y_i^0(s),\tilde y_i(s))- h(0,\xb)$',  one can derive that    $h(\tilde y^0_i(s),\tilde y_i(s))\le 0$  for all  $s\in[0,\bar s]$, for every $i$ large enough, so proving the validity of {\bf (H6)},(i).  Finally, from condition  \eqref{cqdsimple} (by adding and subtracting `$\zeta_0\cdot  (\tilde\w^0_i(s))^d 
 + \zeta\cdot  \F(0,\xb, (\tilde\w^0_i, \tilde\w_i, \tilde\alpha_i)(s))$' and by taking  $(\hat \w^0_i, \hat \w_i,\hat\alpha_i)=(\hat w^0,\hat w,\hat\alpha)$)   we get condition  {\bf (H6)},(iii), possibly reducing $\tilde\delta$, for all $i$ large enough.
 \end{proof} 

\begin{theo}\label{Th12imp}  Let $\bar z:=(\bar S, \bar\w^0,\bar\w,\bar\alpha, \bar{y}^0,\bar y,\bar\nu)$ be a feasible   process for the impulsive extension $(P^*)$, and suppose that at $\bar z$ there is a local infimum gap. If hypothesis  {\bf (H5)} is verified, then $\bar z$ is an abnormal extremal.  If  hypothesis  {\bf (H6)} is also satisfied, then $\bar z$ is a nondegenerate abnormal extremal.
  \end{theo}
  \begin{proof}  It is sufficient to show that hypothesis {\bf (H5)} allows the application of Theorem \ref{Th12FT}, (i), while assuming {\bf (H5)}-{\bf (H6)}, Theorem \ref{Th12FT}, (ii) is applicable. To this aim, we observe that hypothesis {\bf (H1)} is trivially verified,  by choosing, e.g.,  $V_i:=\{(w^0,w)\in V: \ \ w^0\ge \frac{1}{i+1}\}$ for every $i\in\N$, while  {\bf (H5)}  yields {\bf (H2)$^*$} directly.  Condition  {\bf (H3)$^*$} easily follows  from  {\bf (H5)}, taking into account the control-polynomial structure  of the dynamics  as regards point (ii). 
  To prove that   {\bf (H5)}-{\bf (H6)} imply condition   {\bf (H4)},  let us consider a sequence $\delta_i\downarrow0$  and for every $i$, define the strict sense control 
$$
(\check\w^0_i, \check\w_i,\check\alpha_i)(s):=\begin{cases} (\tilde \w^0_i, \tilde \w_i,\tilde\alpha_i)(s) \,\quad\qquad\qquad\qquad\text{if $s\in[0,\bar s]$,} \\
(\bar \w^0, \bar \w,\bar\alpha)(s) \ \ \,\quad\qquad\qquad\qquad\text{if $s\in]\bar s, \bar S]$ and $\bar \w^0(s)>0$,} \\
(\delta_i,\sqrt[d]{1-\delta_i^d} \,\bar \w (s),\bar\alpha (s)) \, \ \qquad\text{if $s\in]\bar s, \bar S]$ and $\bar \w^0(s)=0$,}
\end{cases} 
$$
where  $(\tilde \w^0_i, \tilde \w_i,\tilde\alpha_i)$ is as in {\bf (H6)}.  By identifying, for every $i$, the strict sense process, say $(\check\w^0_i, \check\w_i,\check\alpha_i, \check y_i^0,\check y_i,\check\nu_i)$,  of $(P^*)$ corresponding to $(\check\w^0_i, \check\w_i,\check\alpha_i)$ with a relaxed process --as we have been doing since the Introduction--,   we derive that conditions {\bf (H4)},(ii), and  {\bf (H4)},(iii)  on   $\check E_i:=\tilde E_i\cup]\bar s,\bar S]\subset[0,\bar S]$, are verified. Condition  {\bf (H4)},(i) follows from well-known continuity properties of the input-output map, associated to the control system in  $(P^*)$. Finally,   in view of   Remark \ref{RH4},(3), {\bf (H6)},(iii)  implies    
{\bf (H4)},(iv), although the controls $(\hat \w_i^0,\hat \w_i, \hat\alpha_i)$ are  extended, not necessarily strict sense,  controls.  \end{proof}
As corollaries, we have:
\begin{theo}\label{Cor_NormImp}  Assume hypothesis {\bf (H5)} and let   $\Psi$ be locally Lipschitz continuous.  
\begin{itemize}
\item[{\rm (i)}]    If a [local]   $\Psi$-minimizer  $\bar z$  for  $(P^*_e)$ is   a normal  $\Psi$-extremal,   then  [at $\bar z$] there is no [local]  infimum gap.
\item[{\rm (ii)}] If $\bar z$ is a  [local]   $\Psi$-minimizer    for  $(P^*_e)$, at which condition {\bf (H6)} is verified  and  it is  a nondegenerate normal $\Psi$-extremal,   then  [at $\bar z$] there is no [local]  infimum gap.
  \end{itemize}
  
\end{theo} 
 
\begin{theo}\label{Cor_ContrImp}   Assume hypothesis {\bf (H5)}. Then, either
\begin{itemize}
\item[{\rm (i)}] $\bar z$ is not isolated in $\Gamma^*$; or
\item[{\rm (ii)}] $\bar z$ is an abnormal extremal [a nondegenerate abnormal extremal,  if condition {\bf (H6)} is verified].
  \end{itemize}
\end{theo}

All the above results can be easily extended  to the case when there is a local infimum gap at some $\bar z$, which is a process of the {\em relaxed} problem associated to  $(P^*_e)$.  We limit ourselves to  establish   gap-abnormality relations for the original problem with respect to its  impulsive extension  $(P^*_e)$,      to illustrate how the results of the present paper apply to general, possibly  not convex,  extensions.

\vsm
In the following example  there is no infimum gap  but this fact cannot be deduced from the normality criterion in Theorem \ref{Cor_NormImp},(i), since the extended minimizer is abnormal. Instead, the absence of  gap is detected by the Theorem \ref{Cor_NormImp},(ii),  as the minimizer is nondegenerate normal.
 \begin{exa}\label{exa}    {\rm Consider the problem
\bel{strictex3}
\left\{
\begin{array}{l}
\qquad \qquad\mbox{minimize }\  -x(1)
\\ 
\mbox{over } (x,v, u)  \in W^{1,1}([0,1];\R^{3}\times\R)\times L^1([0,1];\R^2)  \text{ satisfying} \\
( \dot x, \dot v)(t) =\Big( f(x(t))+g_1(x(t))\, u^1 (t) +g_2(x(t))\, u^2(t) , \,\,  \left|u(t)\right| \Big) \\ 
(x,v)(0)=((1,0,0), 0), \\ 
x(t)\in\Omega \  \ \forall t\in[0,1], \  v(1)   \leq 2 ,\
   x(1)\in\T ,
\end{array}
\right.
\eeq
in which 
$
\Omega:=[-1,1]^3$, $
 \T :=  [-1,0]\times[0,1]^2$,
 and
$$
\ds g_1(x):=\left(\begin{array}{l}  1 \\   0\\0\end{array}\right),   \quad \ds g_2(x):=\left(\begin{array}{l} \ \ 0 \\ -1\\-x^1\end{array}\right) , \quad  \ds f(x):=\left(\begin{array}{l} \ 0 \\  x^2x^3\\ \ 0\end{array}\right)  \quad \forall  x\in \R^3\,.
$$
Here,  $W=\{(\w^0,w)\in\R_{\ge0}\times\R^2: \ \ w^0+|w|=1\}$,  $V=\{(\w^0,w) \in W: \ \ w^0>0\}$, and the associated extended  problem is
\[
\left\{
\begin{array}{l}
\qquad\qquad \mbox{minimize } \  -y^1(S)
\\ 
\mbox{over } S>0, \ (y^0, y,\nu, \w^0,\w^1,\w^2) \in W^{1,1}([0,S];\R\times \R^3\times\R)\times\W(S)   \quad
\mbox{satisfying }\\
(\dot y^0, \dot y, \dot \nu)(s) = \Big( \w^0(s), \,\, f(y(s))\w^0(s)+g_1(y(s))\,\w^1(s) +g_2(y(s))\,\w^2(s), \,\, \left|\w(s)\right| \Big) \\
 (y^0, y, \nu)(0)=(0, (1,0,0),0) \\
y(s)\in\Omega  \ \ \forall s\in[0,S],
 \  \ (y^0(S), y(S),\nu(S))\in\{1\}\times \T\times]-\infty,2] .
\end{array}
\right.
\]
As it is easy to see, an extended   minimizer  is  given by the following feasible extended   process $\bar z:=(\bar S, \bar\w^0 ,\bar\w,  \bar y^0 ,\bar y , \bar\nu )$, where
\[
\begin{array}{c}
\bar S=2, \qquad (\bar\w^0,\bar\w)=(\bar\w^0,\bar\w^1,\bar\w^2)  =(1,0,0)\chi_{_{[0,1]}}  +(0,-1,0)\chi_{_{]1,2]}}, \\
(\bar y^0,\bar y,\bar\nu) =(\bar y^0,\bar y^1,\bar y^2,\bar y^3,\bar\nu)  =(s,1,0,0,0 )\chi_{[0,1]} +(1,2-s,0,0,s-1)\chi_{[1,2]} \,.
\end{array}
\]

\noindent From  the maximum principle \cite[Thm. 1.1]{FM120}, $\bar z$ is a $\Psi$-extremal  accordingly to Definition \ref{extremalImp}. Hence,  there exist a set of multipliers $(p_0 ,p,\pi, \gamma,\mu)$ and functions $(m_0,m)$  with  $\pi=0$,  since $\nabla_v\Psi\equiv0$ and  $\bar\nu(2)=1<2$,   $m_0\equiv 0$,  as the state constraint does not depend on time, and $\mu([0,2])=\mu([0,1])$. Moreover, for every $s\in[0,1]$ the fact that $\bar y(s)\in\Omega$ is equivalent to $h(\bar y(s))\le0$, with $h(x^1,x^2,x^3):=x^1-1$, so that the  condition  $m(s)\in\partial^>_xh(\bar y(0))$ $\mu$-a.e.  yields  $m(s)=(1,0,0)$ $\mu$-a.e. in $[0,1]$.  By the adjoint equation,  it follows that  the path $(p_0,p)=(p_0,p_1,p_2,p_3)\equiv(\bar p_0,\bar p_1,\bar p_2,\bar p_3)$ is constant. 
From  the transversality condition 
\[
-(q_0,q_1,q_2,q_3)(2)\in\gamma\{(0,-1,0,0)\}+\R\times N_{\T}(0,0,0), 
\]
where $q_0\equiv\bar p_0$,   and $q(s)=(\bar p_1+\mu([0,1]), \bar p_2,\bar p_3)$ for all $s\in]1,2]$, we derive that $\bar p_0$, $\bar p_1\in\R$, $\bar p_2$, $\bar p_3\ge0$, and $q_1(2)=\bar p_1+\mu([0,1])=\gamma -\alpha_1 $ with $\alpha_1\ge0$.  The maximality condition in $]1,2]$ implies that $\bar p_2=\bar p_3=0$. In particular, from  the relations 
\[
\begin{array}{l}
 \ds\max_{w^1\in[-1,1]}\left\{ q_1(s)w^1\right\}\chi_{[0,1]}(s)=\bar p_0\chi_{[0,1]}(s) =0, \quad -q_1(s)\chi_{]1,2]}(s)=0,
 \end{array}
\]
  we also deduce that $\bar p_0=0$,   $q_1(s)=\bar p_1+\mu([0,s[)=0$ for a.e. $s\in[0,1[$,  and $q_1(s)=\bar p_1+\mu([0,1])=\gamma -\alpha_1=0 $ for every $s\in]1,2]$. In particular, $q(s)=0$ for a.e. $s\in[0,2]$,  $\mu([0,s[)=-\bar p_1$ for a.e. $s\in[0,1]$ implies that ($\bar p_1\le0$ and) $\mu=-\bar p_1\mu(\{0\})$, while the last relation yields that $\gamma=\alpha_1$. 

\vsm
It is immediate to see that the set of degenerate multipliers $(p_0,p,\gamma,\mu)$ with $p_0=p_2=p_3=0$, $p_1=-1$,   $\mu=\delta_{\{0\}}$, and $\gamma=0$   meets all the conditions of  the maximum principle. So,   $\bar z$ is an abnormal extremal.
However, since  $\bar w^0>0$ for a.e. $s\in[0,1]$ and the control $(\hat \w^0,\hat \w)=(\hat \w^0,\hat \w^1,\hat\w^2)\equiv(0, -1,0)$
verifies \eqref{cqdsimple}, from Lemma \ref{LH6} it follows that condition {\bf (H6)} is satisfied. Therefore,  in view of Theorem   \ref{Cor_NormImp}, (ii),  to deduce that  there is no infimum gap  it is enough to observe that $\bar z$ is nondegenerate normal, namely, that $\gamma\ne0$ for all sets of multipliers as above, which in addition verify  
$$
\mu(]0,2])+\|q\|_{L^\infty}+\gamma\ne0.
$$
This is true, since   the previous calculations imply that $\|q\|_{L^\infty}=0$ and $\mu(]0,2])=0$.}
\end{exa}

\section{Proofs of Theorems    \ref{Th1}, \ref{Th2}}\label{S4}
Preliminarily, let us observe that, since  the proofs  involve only relaxed and extended processes with   trajectories  close to  the reference   trajectory $(\bar \xi,\bar y)$ and the controls assume values in   compact sets,  using standard  cut-off techniques   we can assume  that  all hypotheses  {\bf (H2)}-{\bf (H3)$'$} are satisfied not only in $\Sigma_\theta$, but in the whole space $\R^{1+n}$. Hence, for any   $(\underline{\w},\underline{\alpha},\lambda)\in  \W^{1+n}\times\A^{1+n}\times \Lambda_n$   there is a unique solution $(\xi,y)[\underline{\w},\underline{\alpha},\lambda]$ to \eqref{Er} defined on $[0,S]$. Similarly, for any $(\w,\alpha)\in\W\times\A$,  we will write $y[\w,\alpha]$ to denote the corresponding solution to \eqref{E}.

\subsection{Proof of Theorem  \ref{Th2}}
  The proof is divided into several steps in which successive sequences of optimization problems are introduced that have as eligible controls only   strict sense controls, and costs that measure how much a process violates the constraints. Using the Ekeland Principle, minimizers are then built for these problems, which converge to the initial isolated process. Furthermore,  applying a   maximum principle  to these approximate problems with reference to the above mentioned minimizers, we obtain in the limit a set of multipliers with $\gamma=0$  {\em and verifying  the strengthened non-triviality conditions \eqref{nondeg}}   for the relaxed problem  with reference to the isolated process $(\underline{\bar\w},\underline{\bar\alpha},{\bar\lambda},\bar y)$.

\vsm
{\it Step 1.}  Define the function $\Phi:\R^{n+1}\to\R$,  given by
$$
\Phi \left(x,z\right) := d_{\T}(x)\vee z  	
$$
and   for any $y\in W^{1,1}([0,  S];\R^n)$,  introduce  the payoff
$$
{\mathcal J}(y):= \Phi \Big(y(S),\,\max_{s \in [0,S]} h(s,y(s))\Big).
$$
Fix a sequence $(\varepsilon_i)_i$ such that $\varepsilon_i  \downarrow 0$.  Let $(\tilde w_i,\tilde\alpha_i, \tilde\lambda_i)_i$ be a control sequence as in hypothesis {\bf (H4)},  such that, eventually passing to a subsequence, for every $i$,  the  corresponding trajectory $( \tilde \xi_i,\tilde y_i)$ of \eqref{Er} verifies
\bel{stimatilde}
\| ( \tilde \xi_i,\tilde y_i) - (\bar\xi,\bar y) \|_{L^{\infty} } \leq \varepsilon_i.
\eeq
For every $i$, let $\rho_i\ge0$ verify
$$
\rho_i^4=\sup\left\{\begin{array}{l} {\mathcal J}(y): \ \  ( w, \alpha,y)\in\Gamma, \ \ 
  \|y - \bar y\|_{L^{\infty} } \le 2\varepsilon_i 
\end{array}\right\}.
$$
By the Lipschitz continuity of $\Phi$, it follows that $\lim_{i\to+\infty}\rho_i^4=0$. Moreover, $\rho_i>0$ for every $i$ large enough, since $\bar z$ is an isolated   process by Proposition \ref{iso}. 

 In the following, as it is clearly not restrictive, we will always assume that the properties valid from a certain index onwards, apply to each index   $i\in\N$. 
By well-known continuity properties of the input-output map $(\w,\alpha)\mapsto y[\w,\alpha]$, for every $\varepsilon_i$ there exists $\delta_i>0$ such that, if $\|\w - \tilde \w_i\|_{L^1 } \leq \delta_i$, then $\| y[\w, \tilde\alpha_i] - \tilde y_i \|_{L^{\infty} } \leq \varepsilon_i$. According to hypothesis {\bf (H1)} and Remark \ref{Rhyp12}, for any $i$ there exist an element  of the sequence $(V_j)_j$, which we denote by $V_{\delta_i}$,  and some $ \mathring \w_i \in \V_{\delta_i}:=L^1([0,S];V_{\delta_i})$ such that  $\|  \mathring \w_i -\tilde \w_i \|_{L^1} \leq \delta_i$. In particular, if we define
$$
\check \w_i(s):= 
\begin{cases}
\tilde \w_i (s)  \ \  \text{a.e. }s\in[0,\rho_i] \\
 \mathring \w_i(s)  \ \  \text{a.e. }s\in]\rho_i,S],
\end{cases}
(\check \alpha_i(s),\check\lambda_i(s)) := (\tilde \alpha_i(s),\tilde\lambda_i(s))  \ \   \text{a.e. }s\in [0,S],
$$
and $
( \check \xi_i, \check y_i) := (\xi,y)[\check 	\w_i,\dots,\check \w_i, \check\alpha_i,\dots,\check\alpha_i, \check{\lambda}_i],
$
then $\| \check \w_i -\tilde \w_i \|_{L^1} \leq \delta_i$ and    $\check y_i$ is a  strict sense trajectory such that
$
\| (\check\xi_i,\check y_i) - (\tilde\xi_i,\tilde y_i) \|_{L^{\infty}} \leq \varepsilon_i. 
$
From \eqref{stimatilde} it follows  that
\bel{conv1}
\| (\check\xi_i,\check y_i) - (\bar\xi,\bar y) \|_{L^{\infty}} \leq 2\varepsilon_i. 
\eeq
Hence, by the very definition of $\rho_i$ we deduce that for any $i$ the process $\check z_i:=(\check \w_i,\check\alpha_i,\check{\lambda}_i,\check\eta_i, \check\xi_i,\check y_i)$, where $\check\eta_i\equiv 0$,  is a $\rho_i^4$-minimizer for the  optimal control problem:
\begin{equation*}
\big(\hat P_i\big)
\begin{cases}
\qquad\qquad\qquad\qquad\qquad\text{Minimize } \,\,\,{\mathcal J}(y) \\
 \text{over the set of control }(\w,\alpha,\lambda,\eta)\in \V_{\delta_i} \times \A \times \Lambda_n^1 \times L^1([0,S]; \{0,1\}), \\
 \text{and trajectories }  \ (\xi,y) \in W^{1,1}([0, S];\R^{1+n} \times\R^{n}), \text{ satisfying} \\
 \dot \xi(s) = \lambda(s) \qquad \text{ a.e. } s \in [0,S]  \\
 \dot y(s) =  \F(s,y, \tilde \w_i, \tilde\alpha_i) + \eta(s) [\F(s,y, \hat \w_i, \hat\alpha_i)-\F(s,y, \tilde \w_i, \tilde\alpha_i)]\,\, \text{a.e.}\, s \in [0, \rho_i]   \\
\dot y(s) = \F(s,y(s), \w(s), \alpha(s))    \quad \text{ a.e. } s \in  ]\rho_i,S]   \\
 (\xi,y)(0)=(0,\xb),
\end{cases}
\end{equation*}
where $(\hat\w_i,\hat\alpha_i)$  is as in hypothesis {\bf (H4)}. 
We  call an element $(\w,\alpha,\lambda,\eta, \xi, y)$ verifying the constraints in $(\hat P_i)$ a {\em  process for problem $(\hat P_i)$} and use ${\Gamma}_i$ to denote the set of such processes. By introducing, for every  $(\w',\alpha',\lambda',\eta', \xi', y')$, $(\w,\alpha,\lambda,\eta, \xi, y)\in\Gamma_i$,  the distance 
\bel{bfd}
\begin{array}{l}
\mathbf{d}((\w',\alpha',\lambda',\eta', \xi', y'),(\w,\alpha,\lambda,\eta, \xi, y)) \\
 \qquad\qquad:=\|\w'-\w\|_{L^1([0,S])}+\ell\{s\in[0,S]: \ (\alpha',\lambda',\eta')(s)\ne(\alpha,\lambda,\eta)(s)\},
\end{array}
\eeq
 we can make $({\Gamma}_i,\mathbf{d})$ a complete metric space.  
Then,  from Ekeland's Principle it follows that there exists a process $z_i:=(\w_i,\alpha_i, \lambda_i,\eta_i,  \xi_i, y_i)\in{\Gamma}_i$,   which is a minimizer of the optimization problem
 \begin{equation*}
\left(P_i\right)
\begin{cases}
 \text{Minimize} \,\,\,\, {\mathcal J}(y) +\rho^2_i \, 
   \int_{0}^{S}  \left[ |\w(s) - \w_i(s)|  +\ell_i(s,\alpha(s), \lambda(s), \eta(s)) \right] \,ds   \\
 \text{over } \ (\w,\alpha,\lambda,\eta, \xi, y)\in\Gamma_i,
\end{cases}
\end{equation*}
where 
$
\ell_i(s,a, \lambda, \eta):= \chi_{\{(a,\lambda,\eta) \neq (\alpha_i(s),\lambda_i(s), \eta_i(s))\} }
$
for any $(s,a, \lambda, \eta)\in[0,S] \times A\times\Lambda_n^1\times\{0,1\}$,
and verifies 
\begin{equation} \label{ek}
\mathbf{d}\big((\w_i,\alpha_i, \lambda_i,\eta_i,  \xi_i, y_i),(\check \w_i,\check\alpha_i,\check{\lambda}_i,\check\eta_i, \check\xi_i,\check y_i) \big)
 \leq \rho^2_i.  
\end{equation}
Thus, by \eqref{conv1} and the continuity of the input-output map associated to the control system \eqref{Er},  it follows that, eventually passing to a subsequence, as $i \to +\infty$,  
 \begin{equation} \label{conv_y}
\left\| \left(\xi_i, y_i \right)-\left(\bar\xi, \bar y \right)  \right\|_{L^{\infty}} \to 0, \quad
\big(\dot \xi_i , \dot  y_i\big) \rightharpoonup \big(\dot{\bar\xi}, \dot{\bar y} \big) \quad \text{ weakly in $L^1$.}
\end{equation}
Furthermore,  hypothesis {\bf (H4)} and \eqref{ek} imply that, for every $i$, there exist some nonempty subset $E_i\subseteq \tilde E_i\subseteq[0,S]$  and some $r_i\in L^1([0,S];\R_{\ge0})$ with $ r_i\ge \tilde r_i$ ($\tilde E_i$, $\tilde r_i$  as in  {\bf (H4)}) such that, as $i\to+\infty$,  $\ell(E_i)\to S$, $\|r_i\|_{L^1}\to0$, and
\begin{equation} \label{conv_mis}
 (\w_i,\alpha_i,\lambda_i)(s)\in \bigcup_{k=0}^n\{(\bar\w^k(s),\bar\alpha^k(s),e^k)\}+(r_i(s),0,0)\B  \quad\text{for a.e. $s\in  E_i$}. 
 \end{equation}
From \eqref{conv_y} and the fact that  $\bar z$ is  isolated, it follows that 
   ${\mathcal J}( y_i)>0$  for all $i$, namely, at least one of the following  inequalities holds true: \footnote{Notice that, to any process $(\w,\alpha,\lambda,\eta, \xi, y)\in\Gamma_i$ it corresponds a strict sense process $(\breve\w,\breve\alpha,\breve y)$, where $\breve y\equiv y$ and $(\breve\w,\breve\alpha)(s)=(\tilde \w_i,\tilde\alpha_i)(s)+\eta(s)(\hat \w_i-\tilde \w_i,\hat \alpha_i-\tilde \alpha_i)(s)$ a.e. $s\in[0,\rho_i]$, $(\breve\w,\breve\alpha)(s)= (\w,\alpha)(s)$ a.e. $s\in]\rho_i,S]$.}
\begin{equation} \label{violazione} 
d_{\T}(y_i(S))>0,  \qquad c_i:=\max_{s \in [0,S]} \,  h(s,y_i(s))  >0.
\end{equation}
\vsm
{\it Step 2.}    For each $i\in\N$, set
$$
\tilde h(s,x,c) := h(s,x) - c \qquad \forall (s,x,c)\in\R^{1+n+1}.
$$
The process   $(z_i,c_i)=(\w_i,\alpha_i, \lambda_i, \eta_i, \xi_i, y_i,c_i)$ turns out to be a minimizer for 
\begin{equation*}
\left(Q_i\right)
\begin{cases}
\text{Minimize} \,\,\,\Phi(y(S),c(S)) + \rho^2_i\,
 \int_{0}^{S}  \left[ |\w(s) - \w_i(s)|  +\ell_i(s,\alpha(s), \lambda(s), \eta(s)) \right] \,ds  \\
 \text{over } \ (\w,\alpha,\lambda,\eta,\xi, y)\in\Gamma_i, \ c\in W^{1,1}([0,S];\R), \ \text{verifying} \\
\dot c(s)=0, \qquad \tilde h(s, y(s),c(s))\le 0 \quad \forall s\in[0,S].
\end{cases}
\end{equation*}
Passing eventually to a subsequence,   we may suppose that 
\vsm
\qquad either  (a)   {\em `$c_i > 0$ for each $i \in \N$'}, or (b)  {\em`$c_i \leq0$ for each $i \in \N$'.}
\vsm

\noindent Case (a).    Preliminarily, we show  that, for every $i$, one has $h(s,y_i(s))< c_i$ for all $s\in[0,\rho_i]$.
  This result is a  straightforward consequence of the following lemma:  
  \begin{lemma}\label{l2}  For every $i\in\N$, one has $h(s,y_i(s))\le 0$ for all $s\in[0,\rho_i]$.
  \end{lemma}
  \begin{proof} From  a standard application of the Gronwall's Lemma one can  deduce  that
there is $\bar C>0$ such that, for every $i$, one has
\bel{estyi}
  |y_i(s)-\tilde y_i(s)|\le \bar C\,\ell(s, \eta_i(\cdot))  \qquad\forall s\in]0,\rho_i],
  \eeq
 where the nondecreasing map $s \mapsto \ell(s,\eta_i(\cdot))$ is as  in \eqref{ell}. 
Fix now $i\in\N$.   By the Lebourgh Mean Value Theorem \cite[Th. 4.5.3]{OptV},  for every $s\in[0,\rho_i]$  there exists $(\zeta_{0_i}^s,\zeta_{_i}^s) \in\partial^c h(s,x_i(s))$ for some $x_i(s)$ in the segment $[y_i(s)\,, \,\tilde y_i(s)]\subseteq\R^{n}$, such that  
 \footnote{Notice that by the boundedness of the dynamics, both $\tilde y_i(s)$ and $y_i(s)$ lay on $\xb + s K_{_\F} \B$. Hence, for $i$ sufficiently large, $s \in [0,\varepsilon]$ and $x_i(s) \in \xb + \varepsilon \B$, where $\varepsilon>0$ is as in Remark \ref{RH4},(4).}
$$
\begin{array}{l}
 h(s,y_i(s))-h(s,\tilde y_i(s)) = \zeta_i^s\cdot(y_i(s)-\tilde y_i(s)) \\
 \,\,\,=\int_{0}^s \zeta_i^s\cdot[ \F(\sigma,y_i,\tilde\w_i,\tilde\alpha_i) -\F(\sigma,\tilde y_i,\tilde\w_i,\tilde\alpha_i)]\,d\sigma \\ \qquad\qquad\qquad+ \int_{0}^s\eta(\sigma)\zeta_i^s\cdot[\F(\sigma,y_i,\hat\w_i,\hat\alpha_i)-     \F(\sigma,y_i,\tilde\w_i,\tilde\alpha_i)] \,d\sigma \\
 \,\,\, \le \int_{0}^s \bar C K_{_h} K_{_\F}\,\ell (\sigma,\eta_i(\cdot) ) d\sigma -\delta \,\ell(s, \eta_i(\cdot)) \le  \ell(s,\eta_i(\cdot))\left( -\delta+ \bar CK_{_h}K_{_\F}\,s\right) \le0, 
\end{array}
$$
where  the last relations  follow from  \eqref{cqd2},    \eqref{estyi}, and the fact that $s\leq\rho_i\downarrow 0$. Finally, condition (\ref{htilde}) implies the thesis.
 \end{proof} 

Our aim is now to  apply the Pontryagin Maximum Principle to problem $(Q_i)$ with reference to the minimizer $(z_i,c_i)$, for which, thanks to Lemma \ref{l2}, the constraint is inactive on $[0,\rho_i]$.  
By standard arguments (see the proof of \cite[Th. 2.2]{FM120}) we deduce that
$
\partial_{t,x,c}^{^{>}} \tilde h(s,y_i(s),c_i) = \partial_{t,x}^{^{>}} h(s,y_i(s)) \times \{ -1\},
$
and that, if
 $(\beta_{y_i} ,\beta_{c_i})\in \partial \Phi(y_i(S), c_i(S) )$,
 then there is some $\sigma_i^1, \sigma_i^2\geq 0$ with $\sigma_i^1+\sigma_i^2=1$, such that
$
\beta_{y_i}\in\sigma_i^1 \, \left(\partial d_{\T}( y_i(S))\cap \partial\B_{n}\right),
$
$
\beta_{c_i}=\sigma_i^2,
$
and $\sigma_i^k=0$, $k\in\{1,2\}$,  when the maximum in $d_{\T}(y_i(S))\vee c_i(S)$ is strictly greater than the $k$-th term in the maximization.
Thus,  the  Maximum Principle  \cite[Th. 9.3.1]{OptV} yields the existence of some multipliers $(p_i,\pi_i) \in W^{1,1}([0, S]; \R^{n+1})$  associated with  $( y_i, c_i)$,  $\mu_i \in NBV^{+}([0,S]; \R)$, $\gamma_i\geq0$, $\sigma^1_i$, $\sigma_i^2\ge0$ with $\sum_{k=1}^2\sigma_i^k=1$, and Borel-measurable, $\mu_i$-integrable functions $m_i: [0,S] \to \R^n$, such that 
\begin{itemize}
\item[(i)$'$] $\| p_i \|_{L^{\infty}}+ \| \mu_i \|_{TV} + \gamma_i + \| \pi_i \|_{L^{\infty}}  =1$;

\item[(ii)$'$]  $-\dot p_i (s) \in {\rm co} \text{ } \partial_{x} \Big\{ q_i(s) \cdot \F( s,y_i,\w_i,\alpha_i)(s) \Big\}$ for a.e.  $s \in [\rho_i, S]$,  \\
and $\dot \pi_i(s) =0$ for a.e.  $s \in [0, S]$;

\item[(iii)$'$]  $-q_i(S)  \in \gamma_i \,\sigma_i^1\,  \left(\partial   d_{\T}(y_i(S))\cap \partial\B_{n}\right) $,  \ 
 $\pi_i(0)=0$, \  $-\pi_i(S) + \int_{[0,S]} \mu_i(d\sigma)=\gamma_i \sigma_i^2$; 
 
 \item[(iv)$'$] $m_i(s) \in \partial_{x}^{^{>}} h\left(s,  y_i(s)\right)$ \qquad $\mu_i$-a.e. $s \in [0, S]$,

\item[(v)$'$] $spt(\mu_i) \subseteq \{ s \text{ : } h\left( s, y_i(s)\right) - c_i = 0 \}\subset  [\rho_i,S]$,

\item[(vi)$_1'$] $\int_{0}^{\rho_i} \eta_i \,  p_i \cdot \left[\F(s,y_i,\hat\w_i,\hat\alpha_i)- \F(s,y_i,\tilde\w_i,\tilde\alpha_i)\right]  \,ds$ \\
\text{ }\qquad$\geq \int_{0}^{\rho_i} \left\{ (1-\eta_i) \, p_i \cdot \left[\F(s,y_i,\hat\w_i,\hat\alpha_i)- \F(s,y_i,\tilde\w_i,\tilde\alpha_i)\right]   -  2 \gamma_i \rho_i^2\right\} ds$; \footnote{By (v)$'$ it follows that $q_i\equiv p_i$ on $[0,\rho_i]$. Notice also that (vi)$_1'$ holds in a more general form, in fact we can replace $1-\eta_i(\cdot)$ in the right hand side with any measurable function $\eta:[0,\rho_i] \to \{0,1\}$. Furthermore, we assume without loss of generality diam$(W) =1$, since $W$ is supposed to be compact, so that $|\w(s) - \w_i(s)|  +\ell_i(s,\alpha(s), \lambda(s), \eta(s)) \leq 2$ for any $s \in [0,S]$.}
 
\item[(vi)$_2'$] $\int_{\rho_i}^{S}  q_i \cdot \F(s,y_i,\w_i,\alpha_i) ds\ge\int_{\rho_i}^{S} \left\{ q_i \cdot  \F(s,y_i,\w,\alpha) - 2\gamma_i \rho_i^2 \right\}ds$  \\
for all $(\w,\alpha,\lambda,\eta)\in \V_{\delta_i} \times \A \times \Lambda_n^1 \times L^1([0,S];\{0,1\}) $;
\end{itemize}
where
\[
q_i(s) := 
\begin{cases}
  p_i(s) + \int_{[0,s[} m_i(\sigma) \mu_i(d\sigma) \qquad\,\,\, s\in[0, S[, \\
  p_i(S) + \int_{[0, S]} m_i(\sigma) \mu_i(d\sigma) \qquad s=S.
\end{cases}
\]
Observe that, for each $i$,  by (ii)$'$ and (iii)$'$  we derive 
$\| \mu_i \|_{TV}= \int_{[0,S]} \mu_i(ds)=\gamma_i \sigma_i^2$ and $\pi_i\equiv0$. 
 Furthermore,   since $\| m_i \|_{L^{\infty}} \leq K_{_h}$,
then  by (iii)$'$ we get
$$
\gamma_i\sigma_i^1 = \left| q_i(S)  \right| \leq  \| p_i \|_{L^{\infty}}+ K_{_h} \| \mu_i \|_{TV}.
$$ 
By summing up these  estimates and  the non-triviality condition (i)$'$,  we get
$$	
2\| p_i \|_{L^{\infty}}+(2+K_{_h})\| \mu_i \|_{TV} + \geq  \gamma_i(\sigma_i^1+\sigma_i^2-1)  +1= 1.
$$
Hence,    scaling the multipliers, we obtain  
$\| p_i \|_{L^{\infty}}+ \| \mu_i \|_{TV}  = 1$ and $\gamma_i \leq \tilde L:= 2+ K_{_h}$.
\vsm
 
Case (b).  Now, $c_i \leq 0$ for each $i$, so that  \eqref{violazione} implies $d_{\T}(y_i(S))>0$. Thus, the process \linebreak $(\w_i, \alpha_i,\lambda_i,\eta_i, y_i,\xi_i,\hat c_i)$ with $\hat c_i := c_i + \hat\varepsilon$,  for $\hat\varepsilon >0$ suitably small  is still a minimizer of problem $(Q_i)$ and, in addition, it verifies  $h(s,y_i(s)) -\hat c_i <0$  for all  $s\in[0,S]$ (namely, the state constraint is inactive on $[0,S]$). 
  Hence, by applying the Maximum Principle for problem $(Q_i)$ with reference to this minimizer we deduce the existence of multipliers  $(p_i,\pi_i) \in W^{1,1}([0, S]; \R^{n+1})$,  which satisfy conditions (i)$'$--(vi)$'$ with $\pi_i\equiv0$, $\mu_i=0$, $\sigma_i^2=0$, and $\gamma_i > 0$.
In this case,   from  (iii)$'$ we get
$0<\gamma_i  = \left|   q_i(S) \right| =\left|   p_i(S) \right| \leq  \| p_i \|_{L^{\infty}}$,
and,  scaling the multipliers appropriately, we obtain 
$\| p_i \|_{L^{\infty}}= 1$ and $\gamma_i \leq 2$ ($\le \tilde L$ as above). 

\vsm
{\it Step 3.} 
For either the case where   $c_i > 0$ for each $i$ or the case where $c_i \leq0$ for each $i $, passing to the limit  as $i\to+\infty$ for suitable subsequences and arguing as in the proof of   \cite[Thm. 2.2, Step. 4]{FM120},  we can deduce the existence of a set of multipliers $p \in W^{1,1}([0,S]; \R^n)$,  $\mu \in NBV^{+}([0,S]; \R)$ and  a Borel-measurable,   $\mu$-integrable  function $m: [0,S] \to \R^{n}$, such that
\bel{risultati_parziali}
\begin{aligned}
& \|p\|_{L^{\infty}} + \| \mu\|_{TV} =1, \qquad & spt(\mu) \subseteq \{ s\in[0,S] : h(s, \bar y(s))=0   \}, \\
& -q(S) \in N_\T(\bar y(S)) \qquad & m(s) \in \partial_x^> h(s, \bar y(s)) \,\,\text{$\mu$-a.e. $s\in[0,S]$} \\
\end{aligned}
\eeq
where
\[
q(s) := 
\begin{cases}
p(s) + \int_{[0,s[} m(\sigma) \mu(d\sigma) \qquad s\in[0,S[ \\
p(S) + \int_{[0,S]}m(\sigma) \mu(d\sigma) \qquad s=S.
\end{cases}
\]
Furthermore,
\bel{conv_p}
q_i \to q \,\,\, \text{in $L^1$}, \qquad p_i\to p\,\,\, \text{in $L^{\infty}$}, \qquad \dot p_i \rightharpoonup \dot p \,\,\,\text{weakly in $L^1$}.
\eeq


Now, let $\Omega_i:= [\rho_i, S] \setminus E_i$, where $E_i$ is as in \eqref{conv_mis}, so that $\ell(\Omega_i) \to S$. Recalling $\partial_x (q\cdot\F) = q \cdot \partial_x^c \F$, by (ii)$'$, \eqref{conv_mis} and {\bf (H3)},(ii) we deduce, for a.e. $s\in\Omega_i$,
$$
\begin{array}{l}
\ds \big(-\dot p_i, \dot \xi_i,\dot y_i\big)(s)  \in \bigcup_{k=0}^n \Big[ \big({\rm co} \text{ } \partial_{x} \big\{ q_i(s) \cdot \F( s,y_i,\bar\w^k,\bar\alpha^k)(s) \big\} \big) \\
\qquad\qquad\,\, \times \{ ( e^k, \F(s,y_i,\bar\w^k, \bar\alpha^k)(s)   )   \} \Big] + \big((1+K_{_h}) \varphi(r_i(s)) , \,0, \,\varphi(r_i(s)) \big)\B  \\[1.5ex] 
\ds \subseteq \bigcup_{k=0}^n \Big[ \big({\rm co} \text{ } \partial_{x} \big\{ q(s) \cdot \F( s,y_i,\bar\w^k,\bar\alpha^k)(s) \big\} \big) \times \{ ( e^k, \F(s,y_i,\bar\w^k, \bar\alpha^k)(s)    )  \} \Big] \\
\qquad\qquad\,\,  + \big((1+K_{_h}) \varphi(r_i(s)) + K_{_\F} |q_i(s) - q(s)|, \,0, \,\varphi(r_i(s)) \big)\B.
\end{array}
$$
 By the properties of $\varphi(\cdot)$, the compactness of $W$,
   and the  Dominated Convergence Theorem we have $\varphi(r_i) \to 0$ in $L^1$ as $i\to \infty$. Hence, all the hypotheses of the Compactness of Trajectories Theorem \cite[Th. 2.5.3]{OptV} are satisfied, so that we can pass to the limit and get 
\[
\begin{array}{l}
\ds \Big(-\dot p, \dot{\bar\xi},\dot{\bar y}\Big)(s) \in \text{{\rm co}} \Big(  \bigcup_{k=0}^n \big[ \big({\rm co} \text{ } \partial_{x} \big\{ q(s) \cdot \F( s,\bar y,\bar\w^k,\bar\alpha^k)(s) \big\} \big) \\
\qquad\qquad\qquad\qquad\qquad\qquad \times \{ ( e^k, \F(s,\bar y,\bar\w^k, \bar\alpha^k) (s)  )  \}  \big]\Big) \ \ \text{ for a.e. $s\in [0,S]$}. 
\end{array}
\]
By the Caratheodory Representation Theorem, there exists a measurable function $\lambda=(\lambda^0,\dots,\lambda^n)\in\Lambda_n$ such that  
\bel{ad_eq2}
\begin{array}{l}
\ds\Big(-\dot p, \dot{\bar\xi},\dot{\bar y}\Big)(s)  \in \sum_{k=0}^n \lambda^k(s) \Big( {\rm co} \text{ } \partial_{x} \left\{ q(s) \cdot \F( s,\bar y,\bar\w^k,\bar\alpha^k)(s) \right\} \\
\text{ }\qquad\qquad\qquad\qquad\qquad \times \{ ( e^k, \F(s,\bar y,\bar\w^k, \bar\alpha^k) (s)  )   \} \Big) \ \ \text{ for a.e. $s\in [0,S]$}. 
\end{array}
\eeq
But now 
$$
\dot{\bar\xi}(s) = \sum_{k=0}^n \lambda^k(s) e^k = \sum_{k=0}^n \bar\lambda^k(s) e^k \qquad \text{a.e. $s\in[0,S]$.} 
$$ 
Therefore, for every $k=0,\dots,n+1$,  $\lambda^k(s)=\bar\lambda^k (s)$ a.e. $s\in[0,S]$ and \eqref{ad_eq} is proved.

Let us prove \eqref{maxham}. Take $(\w,\alpha) \in \W\times \A$, by {\bf (H1)} and Remark \ref{Rhyp12} there exists a sequence $(v_i)_i \in \V$ such that $v_i \in \V_{\delta_i}$ for any $i$ and $\| \w- v_i \|_{L^1} \leq \delta_i \downarrow 0$. By (vi)$_2'$, we deduce that, for any $i$, one has
\[
\int_0^{S}  q_i(s) \cdot \dot y_i(s) \chi_{[\rho_i,S]}(s) ds\ge\int_{0}^{S} \left\{ q_i(s) \cdot  \F(s,y_i,v_i,\alpha)(s) - 2\gamma_i \rho_i^2 \right\}\chi_{[\rho_i,S]}(s)  ds.
\]
Passing  to the limit and using  \eqref{conv_y}, \eqref{conv_p} in the left hand side and  
 the Dominated Convergence Theorem in the right hand side, we  obtain
\[
\int_0^{S}  q(s) \cdot \dot{\bar y}(s) \,ds\ge\int_{0}^{S} q(s) \cdot  \F(s,\bar y(s),\w(s),\alpha(s))\, ds.
\]
Since this last relation holds for any selector  $(\w,\alpha) \in \W\times\A$, by a measurable selection theorem we can conclude that 
\bel{provvisorio_ham}
q(s) \cdot \dot{\bar y}(s) =  \max_{(w,a) \in W \times A} q(s) \cdot \F(s,\bar y(s), w,a) \qquad \text{a.e. $s\in[0,S]$.}
\eeq
Finally, \eqref{provvisorio_ham} trivially implies \eqref{maxham}. 
Thus $\bar z$ is an abnormal extremal. To prove that it is in fact a nondegenerate  abnormal extremal,  it remains to show that the above multipliers verify the strengthened non-triviality condition
\bel{ntnondeg}
   \| q \|_{L^{\infty}}+ \mu(]0,S])\ne0.
   \eeq
To this aim, assume by contradiction that
$\| q \|_{L^{\infty}}+ \mu(]0,S])=0.$ 
Then, the non-triviality condition \eqref{risultati_parziali} yields that $\mu(\{0\})\ne0$ and $p\equiv-\mu(\{0\})\zeta$ for some $\zeta \in\partial_x^> h(0,\xb)$.  For every $i$, by  the maximality condition (vi)$_1'$ and condition (\ref{cqd2}),  it follows that 
$$
\begin{array}{l}
\ds0\ge  \int_{0}^{\rho_i} (1-2\eta_i)\,  p \cdot \left[\F(s,y_i,\hat\w_i,\hat\alpha_i)- \F(s,y_i,\tilde\w_i,\tilde\alpha_i)\right]   \,ds \\
\ds\qquad +\int_{0}^{\rho_i} \left\{ (1-2\eta_i) \,(p_i-p) \cdot \left[\F(s,y_i,\hat\w_i,\hat\alpha_i)- \F(s,y_i,\tilde\w_i,\tilde\alpha_i)\right] - 2 \gamma_i \rho_i^2\right\}  \,ds \\[1.5ex]
\ds\ \ \ge  \int_0^{\rho_i} p \cdot \left[\F(s,y_i,\hat\w_i,\hat\alpha_i)- \F(s,y_i,\tilde\w_i,\tilde\alpha_i)\right] \chi_{\{\sigma: \ \eta_i(\sigma)=0\}}(s)  \,ds \\
\qquad \ds- \int_0^{\rho_i}  p \cdot \left[\F(s,y_i,\hat\w_i,\hat\alpha_i)- \F(s,y_i,\tilde\w_i,\tilde\alpha_i)\right] \chi_{\{\sigma: \ \eta_i(\sigma)=1\}}(s)\,ds \\
\qquad \qquad -2\rho_i(K_{_\F}\|p_i-p\|_{L^\infty} +\tilde L\rho_i^2 )\\[1.5ex]
\ds \ \ \ge  \mu(\{0\})\,\delta\,\ell(\rho_i, 1-\eta_i(\cdot))
 - 2K_{_\F}K_{_h} \,\ell(\rho_i, \eta_i(\cdot)) -2\rho_i(K_{_\F}\|p_i-p\|_{L^\infty} +\tilde L\rho_i^2 )\\[1.5ex]
\ds \ \ \ge \rho_i\, \big[\mu(\{0\})\,\delta -\mu(\{0\})\,\delta\,\rho_i-2K_{_\F}K_{_h}\rho_i-2K_{_\F} \|p_i-p\|_{L^\infty} -2\tilde L\rho_i^2\big]>0,
\end{array}
$$  
where   we use the facts that $\ell(\rho_i,\eta_i(\cdot))\le \rho_i^2$ and consequently  $\ell(\rho_i, 1-\eta_i(\cdot))\ge \rho_i- \rho_i^2$, which follow from  \eqref{ek}.  Thus, we obtain the desired contradiction.

\subsection{Proof of Theorem \ref{Th1}}\label{SubTh1}
 Preliminarily, observe that hypothesis {\bf (H3)} can be reduced to {\bf (H3)$'$}. We can clearly take $k \geq 1$ in assumption {\bf (H3)}, but actually we may (and we do) assume without loss of generality  $k\equiv1$. Indeed, reasoning as in \cite[Sec. 2]{C05}, we can introduce the   time change $t=\sigma(s):= \int_{0}^s k(\tau) d\tau$, so that $(\hat{\underline \w}, \hat{\underline\alpha}, \hat\lambda,\hat y) := (\underline \w,\underline\alpha, \lambda,y) \circ \sigma^{-1}$ is a process for the {\em transformed} problem, with dynamics $\hat F = \frac{1}{k}  \sum_{j=0}^n\lambda^j\, \F(s,y,\w^j,\alpha^j) $, verifying  {\bf (H3)} for $k\equiv 1$,  and   interval $[0, \sigma(S)]$,  if and only if $(\underline \w,\underline\alpha,  \lambda, y)$ is a process for the relaxed problem. Furthermore,  the transformed process, say $\hat z$,  corresponding to $\bar z:=(\underline{\bar\w},\underline{\bar\alpha},{\bar\lambda},\bar\xi,\bar y)$ is isolated for the transformed problem, and if $\hat z$ is an abnormal extremal for the transformed problem for some $(\hat p,0,\hat\mu,\hat m)$  as in Definition  \ref{Dnormal}, then $\bar z$ is an   abnormal extremal  with $(p,\gamma,\mu,m)$ verifying $p = \hat p \circ \sigma$, $\gamma=0$,  $d\mu = k \,d\hat\mu$ and $m= \hat m\circ \sigma$. 
\vsm
First of all, we notice that $(\bar\xi, \bar y)$  is a solution of the differential inclusion 
$$
\big(\dot\xi, \dot y \big)(s) \in \text{{\rm co} } \bigcup_{k=0}^n \{ (e^k,  \F(s, y(s), \bar \w ^k (s), \bar \alpha^k(s)) )    \} \qquad \text{a.e. $s\in [0,S]$}.
$$ 
Let us fix a sequence $\varepsilon_i\downarrow 0$. By the Relaxation Theorem \cite[Th. 2.7.2]{OptV}, there exists a sequence of extended processes $(\bar\w_i,\bar\alpha_i,\bar\lambda_i)(s)\in\bigcup_{k=0}^n\{(\bar\w^k(s),\bar\alpha^k(s),e^k)\}$ for a.e. $s \in [0,S]$ such that, for any $i$,  the corresponding trajectory $(\bar \xi_i, \bar y_i) := (\xi,y)[\bar\w_i,\dots,\bar\w_i,\bar\alpha_i,\dots,\bar\alpha_i, \bar\lambda_i]$ satisfies 
$$
\| (\bar\xi_i,\bar y_i) - (\bar \xi,\bar y) \|_{L^{\infty}} \leq \varepsilon_i. 
$$
Let $\mathcal{J}(\cdot)$, and $(\rho_i)_i$ be as in the proof of Theorem \ref{Th2} and $(\delta_i)_i$ such that for every $\w \in \W$ with $\| \w - \bar\w_i\|_{L^1} \leq \delta_i$, then $\| y[\w,\bar\alpha_i] - \bar y_i  \|_{L^{\infty}} \leq \varepsilon_i$. Then, thanks to hypothesis {\bf (H1)}, for any $i$ there exists $ \check \w_i \in \V_{\delta_i}$ such that $\| \check \w_i -\bar \w_i \|_{L^1} \leq \delta_i$. As a consequence, if we define $\check\alpha_i \equiv \bar\alpha_i$, $\check \lambda_i \equiv \bar\lambda_i$ and $\check y_i = y[\check\w_i,\check\alpha_i ]$, then $\check y_i$ is a strict sense trajectory that satisfies $\| \check y_i  - \bar y_i \|_{L^{\infty}} \leq \varepsilon_i$ and 
$$
(\check\w_i,\check\alpha_i,\check\lambda_i)(s)\in\bigcup_{k=0}^n\{(\bar\w^k(s),\bar\alpha^k(s),e^k)\} + (\check r_i(s),0,0)\B \qquad \text{a.e. $s\in[0,S]$,}
$$
for some measurable sequence $\check r_i \to 0$ in $L^1$. Therefore, the process $(\check \w_i,\check\alpha_i,\check{\lambda}_i, \check\xi_i,\check y_i)$ is a $\rho_i^4$-minimizer for the optimal control problem
\begin{equation*}
\left(\check P_i\right)
\begin{cases}
\qquad\qquad\qquad\qquad\qquad\text{minimize } \,\,\,{\mathcal J}(y) \\
 \text{over }(\w,\alpha,\lambda, y,\xi)\in \V_{\delta_i} \times \A \times \Lambda_n^1 \times W^{1,1}([0, S];\R^{n+n}),\text{ satisfying}\\
 \ds (\dot\xi, \dot y)(s) = \big(\lambda(s), \,\,  \F(s,y(s), \w(s), \alpha(s)) \big) \ \text{ a.e. } s \in [0,S]  \\
 (\xi,y)(0)=(0,\xb).
\end{cases}
\end{equation*}
From now on, except for minor obvious changes, the proof proceeds as the proof of Theorem \ref{Th2} and is actually simpler, since we disregard the nondegeneracy issue. 

\end{document}